\documentclass[a4paper]{article}

\usepackage[T1]{fontenc}
\usepackage{amsmath,amssymb,amsthm,mathrsfs}
\usepackage{cite,mathtools,enumitem}

\usepackage[pdftex,bookmarksopen=true,bookmarks=true,unicode,setpagesize]{hyperref}
\hypersetup{colorlinks=true,linkcolor=black,citecolor=black,urlcolor=blue}

\let\equation=\gather
\let\endequation=\endgather

\allowdisplaybreaks[3]
\numberwithin{equation}{section}

\makeatletter
\renewcommand*{\@fnsymbol}[1]{\ensuremath{\ifcase#1\or 1\or 2\or
   3\else\@ctrerr\fi}}
\makeatother

\newtheorem{theorem}{Theorem}[section]
\newtheorem{corollary}[theorem]{Corollary}
\newtheorem{lemma}[theorem]{Lemma}
\newtheorem{proposition}[theorem]{Proposition}
\theoremstyle{definition}
\newtheorem{definition}[theorem]{Definition}
\theoremstyle{remark}
\newtheorem{remark}[theorem]{Remark}

\newcounter{assum}
\newenvironment{assum}[1][]{\ifx\newenvironment#1\newenvironment\refstepcounter{assum}\fi\equation\tag{A\theassum#1}}{\endequation}

\DeclareMathOperator*{\esssup}{esssup}

\newcommand{\R}{\mathbb{R}}
\newcommand{\X}{{\mathbb{R}^d}}
\newcommand{\N}{\mathbb{N}}
\newcommand{\eps}{\varepsilon}
\newcommand{\la}{\lambda}
\newcommand{\La}{\Lambda}

\newcommand{\kl}{\varkappa_{{\ell}}}
\newcommand{\kn}{\varkappa_{{n\ell}}}
\newcommand{\kam}{\varkappa^-}
\newcommand{\kap}{\varkappa^+}
\newcommand{\x}{\mathcal{X}}
\newcommand{\1}{1\!\!1}

\newcommand{\A}{{\mathfrak{a}}}

\newcommand{\Buc}{C_{ub}(\X)} 
\newcommand{\M}{{\mathcal{M}_\theta}(\R)}
\newcommand{\Utheta}{{U_\theta}}
\newcommand{\Xinf}{\mathcal{X}_{\infty}}
\newcommand{\locun}{\xRightarrow{\,\mathrm{loc}\ }}
\newcommand{\Y}{\mathcal{U}}

\DeclareFontFamily{U}{mathx}{\hyphenchar\font45}
\DeclareFontShape{U}{mathx}{m}{n}{
      <5> <6> <7> <8> <9> <10>
      <10.95> <12> <14.4> <17.28> <20.74> <24.88>
      mathx10
      }{}
\DeclareSymbolFont{mathx}{U}{mathx}{m}{n}
\DeclareFontSubstitution{U}{mathx}{m}{n}
\DeclareMathAccent{\widecheck}{0}{mathx}{"71} 

\title{Doubly nonlocal Fisher--KPP equation: Existence and properties of traveling waves}

\author{Dmitri Finkelshtein\thanks{Department of Mathematics,
Swansea University, Singleton Park, Swansea SA2 8PP, U.K. ({\tt d.l.finkelshtein@swansea.ac.uk}).} \and Yuri Kondratiev\thanks{Fakult\"{a}t
f\"{u}r Mathematik, Universit\"{a}t Bielefeld, Postfach 110 131, 33501 Bielefeld,
Germany ({\tt kondrat@math.uni-bielefeld.de}).} \and Pasha Tkachov\thanks{Gran Sasso Science Institute, Viale Francesco Crispi, 7, 67100 L'Aquila AQ, Italy ({\tt pasha.tkachov@gssi.it}).}}

\begin{document}
\maketitle

\begin{abstract}
We consider a reaction-diffusion equation with nonlocal anisotropic diffusion and a linear combination of local and nonlocal monostable-type reactions in a space of bounded functions on $\X$. Using the properties of the corresponding semiflow, we prove the existence of monotone traveling waves along those directions where the diffusion kernel is exponentially integrable. Among other properties, we prove continuity, strict monotonicity and exponential integrability of the traveling wave profiles.



\end{abstract}

\section{Introduction} 
\subsection{Description of equation}
We will study the following initial value problem
\begin{equation}
    \begin{cases}
      \dfrac{\partial u}{\partial t}(x,t)=\kap(a^{+}*u)(x,t)-m u(x,t)-u(x,t) (Gu)(x,t),\quad t>0,\\[3mm]
          u(x,0)=u_0(x),
    \end{cases}
\label{eq:basic}
\end{equation}
with 
\begin{equation}\label{eq:defofG}
    (Gu)(x,t):=\kl u(x,t) + \kn ( a^-*u )(x,t),
\end{equation}
which generates a semi-flow $u(\cdot,0)\longmapsto u(\cdot,t)$, $t>0$, in a class of bounded nonnegative functions on $\X$, $d\geq1$. Here 
$\kap, m>0$ and $\kl, \kn\geq0$  are constants, such that
\begin{equation}\label{eq:kaminusissum}
    \kam := \kl+\kn > 0;
\end{equation}
and the functions  $0\leq a^\pm \in L^{1}(\X)$ are probability densities, i.e.
\begin{equation}\label{normed}
  \int_{\X }a^+(y)dy=\int_{\X }a^-(y)dy=1.
\end{equation}
The symbol $*$ denotes the convolution with respect to the space variable, i.e. 
\[
(a^\pm*u)(x,t):=\int_{\X }a^\pm(x-y)u(y,t)dy.
\]

The solution $u=u(x,t)$ describes the local density of a species at the point $x\in\X$ at the moment of time $t\geq0$. The individuals of the species spread over the space $\X$ according to the dispersion kernel $a^+$ and the fecundity rate $\kap$. The individuals may die according to both constant mortality rate $m$ and density dependent competition, described by the rate $\kam$. The competition may be \emph{local}, when the density $u(x,t)$ at a point $x$ is influenced by itself only, with the rate $\kl$, or \emph{nonlocal}, when the density $u(x,t)$ is influenced by all values $u(y,t)$, $y\in\X$, averaged over $\X$ according to the competition kernel $a^-$ with the rate $\kn$.  

For the case $\beta:=\kap -m>0$, the equation \eqref{eq:basic} can be rewritten in the reaction-diffusion form 
\begin{equation}\label{eq:RDform}
\begin{aligned}
    \dfrac{\partial u}{\partial t}(x,t)&=\kap \int_\X a^+(x-y)\bigl(u(y,t)-u(x,t)\bigr)\,dy\\&\quad +u(x,t)\bigl(\beta - (Gu)(x,t)\bigr).
\end{aligned}
\end{equation}
The first summand here describes a non-local diffusion generator, see e.g. \cite{AMRT2010} (also known as the generator of a continuous time random walk in $\X$ or of a compound Poisson process on $\X$). As a result, the solution $u$ to \eqref{eq:RDform} may be interpreted as a density of a species which invades according to a nonlocal diffusion within the space $\X$ meeting a reaction $Fu:=u(\beta -Gu)$; see e.g. \cite{Fif1979,Mur2003,SK1997}.

The non-local diffusion in reaction-diffusion equations first appeared (for the case $d=1$) in the seminal paper \cite{KPP1937} by  Kolmogorov, Petrovsky and Piskunov, to~describe a dynamics where individuals move during the time between birth and reproduction meeting a local reaction $Fu=f(u)=u(1-u)^2$. Using a diffusive scaling, the equation in \cite{KPP1937} was informally transformed to
\begin{equation}\label{kpp}
 \dfrac{\partial u}{\partial t}(x,t)=\alpha \Delta u(x,t)+f\bigl(u(x,t)\bigr),
\end{equation}  
where $\Delta$ denotes the Laplace operator, $\alpha>0$. The choice of the local reaction $f(u)=u(1-u)^2$ was motivated by a discrete genetic model. The equation \eqref{kpp} was studied in \cite{KPP1937}, for a class of reactions which includes also, in particular,  
\[
  f(u)=u(1-u)
\]
that corresponds to $\kn=0$, $\kl=1$, $\beta=1$ in \eqref{eq:defofG} and \eqref{eq:RDform}. The latter reaction was early considered by Fisher in \cite{Fis1937} for another genetic model. The Fisher--KPP equation \eqref{kpp} has been actively studied and generalized since then, see e.g. \cite{AW1978,HR2016,Saa2003} and references therein. 

Later, the equation \eqref{eq:RDform} with  local $G$, i.e.~with $\kn=0$ in \eqref{eq:defofG}, was considered in \cite{Sch1980} (motivated by an analogy to Kendall's epidemic model) and has been actively studied  in the last decade, see e.g. \cite{CD2007,Yag2009,AGT2012,CDM2008a,Gar2011,LSW2010,SLW2011,CY2017} for $d=1$ and \cite{CDM2008,SZ2010} for $d\geq1$.

The equation \eqref{eq:RDform} with pure nonlocal $G$, i.e.~with $\kl=0$, $\kam =\kn$ in \eqref{eq:defofG}, first appeared, for the case $\kap  a^+=\kam  a^-$, $m=0$, in \cite{Mol1972a,Mol1972}. Next, it was derived from a lattice `crabgrass model', for the case $\kap  a^+=\kam  a^-$, $m>0$ in \cite{Dur1988} and latter considered in \cite{PS2005}. 

Note also that, in the pure nonlocal case $\kl=0$, the microscopic (individual-based) model of spatial ecology corresponding to the equation \eqref{eq:basic} was proposed by Bolker and Pacala in \cite{BP1997}. In~this case, the equation \eqref{eq:basic} can be rigorously derived in a proper scaling limit of the corresponding multi-particle evolution; see \cite{FM2004} for integrable species densities and \cite{FKK2011a,FKKozK2014} for bounded ones.

In the present paper, we consider a unified approach to both local and nonlocal competition terms in \eqref{eq:basic}. 

\subsection{Description of results}
	Clearly, $u\equiv0$ is a constant stationary solution to \eqref{eq:basic}. We will assume in the sequel that
\begin{assum}\label{as:chiplus_gr_m}
	\kap >m.
\end{assum}
Then the equation \eqref{eq:basic} has the unique positive constant stationary solution $u\equiv\theta$, where
\begin{equation}\label{theta_def}
  \theta:=\frac{\kap -m}{\kam }>0.
\end{equation}

Our primary object of investigation are monotone traveling waves, which connect $0$ and $\theta$. Let $\M$ denote the set of all decreasing and right-continuous functions $f:\R\to[0,\theta]$. By a (monotone) traveling wave solution to \eqref{eq:basic} in a direction $\xi\in S^{d-1}$ (the unit sphere in $\X$), we will understand a solution of the form 
\begin{equation}\label{eq:deftrw}
      \begin{gathered}
      u(x,t)=\psi(x\cdot\xi-ct),  \quad t\geq0, \ \mathrm{a.a.}\ x\in\X, \\
      \psi(-\infty)=\theta, \qquad \psi(+\infty)=0,
      \end{gathered}
\end{equation}
where $c\in\R$ is called the speed of the wave and the function $\psi\in\M$
is called the profile of the wave.
Here and below $x\cdot \xi$ denotes the scalar product in $\X$. Such solutions are also called in literature as traveling planes, see e.g. \cite{Eva2010}. 

Define the function
\begin{equation}\label{diffofkernels}
  J_\theta(x):=\kap a^+(x)-\kn \theta a^-(x),\quad x\in\X.
\end{equation}
For a fixed $\xi\in S^{d-1} $, we introduce the following assumptions:
\begin{assum}\label{as:aplus_gr_aminus-intro}
      \int_{\{x\cdot\xi=s\}} J_\theta(x)\,dx\geq0 \quad \text{for a.a. } s\in\R,
\end{assum}
cf. \eqref{apm1dim}, \eqref{acheckpos} below, and 
\begin{assum}\label{aplusexpint1}
  \text{there exists} \ \mu=\mu(\xi)>0 \ \text{such that} \ \int_\X a^+(x) e^{\mu \, x\cdot \xi}\,dx<\infty.
\end{assum}
Stress that assumption \eqref{as:aplus_gr_aminus-intro} is redundant for the case of the local $G$, when $\kn=0$, i.e. for the case of the local reaction $Fu=f(u)=u(\beta -\kl u)$.

We will also use the following counterpart of \eqref{as:aplus_gr_aminus-intro}: there exist $\rho,\delta>0$ (depending on $\xi$), such that
          \begin{assum}\label{as:aplus-aminus-is-pos1d}
    \int_{\{x\cdot\xi=s\}} J_\theta(x)\,dx \geq\rho  \text{ \ for a.a. } |s|\leq \delta.
\end{assum}
The following theorem is the main result of the article.
\begin{theorem}\label{thm:trwexist}
Let $\xi\in S^{d-1}  $ be fixed, and suppose that \eqref{as:chiplus_gr_m}, \eqref{as:aplus_gr_aminus-intro}, \eqref{aplusexpint1} hold.
Then there exists $c_*(\xi)\in\R$, such that for any $c<c_*(\xi)$,  a traveling wave solution to \eqref{eq:basic} of the form \eqref{eq:deftrw} with $\psi\in\M$ does not exist; whereas, for any $c\geq c_*(\xi)$, 
\begin{enumerate}[label={\arabic*})]
          \item there exists a traveling wave solution to \eqref{eq:basic} with the speed $c$ and a profile $\psi\in\M$ such that  \eqref{eq:deftrw} holds;
          \item if $c\neq0$, then the profile $\psi\in C_b^\infty(\R)$ (the class of infinitely many times differentiable functions on $\R$ with bounded derivatives); if $c=0$ (in the case $c_*(\xi)\leq 0$), then $\psi\in C(\R)$;
       		\item there exists $\mu=\mu( c, a^+, \kam, \theta)>0$ such that
              \begin{equation}\label{eq:trwexpint}
                \int_\R\psi(s)e^{\mu s}\,ds<\infty;
              \end{equation}
          \item let \eqref{as:aplus-aminus-is-pos1d} hold, then the profile $\psi$ is a strictly decreasing function on $\R$;
					\item let \eqref{as:aplus-aminus-is-pos1d} hold, then, for any $c\neq0$, there exists $\nu>0$, such that $\psi(t)e^{\nu t}$ is a strictly increasing function.
\end{enumerate}
\end{theorem}

\begin{remark}
The last two items of Theorem~\ref{thm:trwexist} will be proven in Propositions~\ref{prop:psidecaysstrictly} and \ref{prop:trw_willbe_incr} below under assumptions weaker than \eqref{as:aplus-aminus-is-pos1d}.
\end{remark}

\begin{remark}
The results of \cite{FKT100-3, FT2017c} show that the assumption \eqref{aplusexpint1} is `almost' necessary to have traveling wave solutions in the equation \eqref{eq:basic}.
\end{remark}

By a solution to \eqref{eq:basic} on $[0,T)$, $T\leq \infty$, we will understand the so-called classical solution, that is a mapping from $[0,T)$ to a Banach space $E$ of bounded functions on $\X$ which is continuous in $t\in[0,T)$, continuously differentiable (in the sense of the norm in $E$) in $t\in(0,T)$, and satisfies \eqref{eq:basic}. The space $E$ is either the space $L^{\infty}(\X)$
of  essentially bounded (with respect to the Lebesgue measure) functions on $\X$ with $\esssup$-norm, or its Banach subspaces $C_b(\X)$ or $\Buc$ of bounded continuous or, respectively, bounded uniformly continuous functions on $\X$ with $\sup$-norm.

Consider, according to \eqref{eq:defofG}, the mapping 
\begin{equation}\label{eq:defofmapG}
    Gu:=\kl u+ \kn a^-*u, \quad u\in E.
\end{equation}
Clearly, $G$ maps $E$ to $E$ and preserves the cone $\{0\leq u\in E\}$. Here and below, all point-wise inequalities for functions from $E$ we will consider, for the case $E=L^\infty(\X)$, almost everywhere only. Moreover, the mapping $G$ is globally Lipschitz on~$E$. In particular, it satisfies the conditions of \cite[Theorem 2.2]{FT2017a} that can be read, in our case, as follows.
  \begin{theorem}[{cf.~\cite[Theorems 2.2, 3.3]{FT2017a}}]\label{thm:existuniq}
Let $0\leq a^\pm\in L^1(\X)$, $m>0$, $\kl,\kn\geq0$ be such that \eqref{eq:kaminusissum} and \eqref{normed} hold.
Then, for any $0\leq u_0\in E$ and for any $T>0$, there exists a unique classical solution $u$ to \eqref{eq:basic} on $[0,T)$. In particular, $u$ is a unique classical solution to \eqref{eq:basic} on $[0,\infty)$.
\end{theorem}

For any $t\geq0$ and $0\leq f\in L^\infty(\X)$, we define
\begin{equation}
  (Q_{t}f)(x):=u(x,t),\qquad \text{a.a. } x\in\X,\label{def:Q_T}
\end{equation}
where $u(x,t)$ is the solution to \eqref{eq:basic} with the initial condition $u(x,0)=f(x)$. 
From the uniqueness arguments and the proof of \cite[Theorems~2.2, 3.3]{FT2017a}, we imemdiately get that $(Q_t)_{t\geq0}$ constitutes a continuous semi-flow on the cone $\{0\leq f\in L^\infty(\X)\}$, i.e. $Q_t$ is continuous at $t=0$ and
\[
  Q_{t+s}f=Q_t(Q_s f), \quad t,s\geq0,
\]
for each $0\leq f\in L^\infty(\X)$.

It can be checked (see Proposition~\ref{prop:statsol} below) that $u\equiv 0$ is an unstable solution to \eqref{eq:basic} and that the following reinforced version of \eqref{as:aplus_gr_aminus-intro},
\begin{equation}\label{as:aplus_gr_aminus}\tag{\ref{as:aplus_gr_aminus-intro}${}'$}
  J_\theta(x)\geq 0,\quad \text{a.a.}\ x\in\X,
\end{equation}
is a sufficient condition to that $u\equiv\theta$ is a uniformly and asymptotically stable solution, in the sense of Lyapunov.

Similarly to above, the assumption \eqref{as:aplus_gr_aminus} is redundant for the case of the local $G$, when $\kn=0$, $Fu=f(u)=u(\beta -\kl u)$.

In \cite[Proposition 5.4]{FT2017a}, we considered properties of the semi-flow $Q_t$ generated by the equation \eqref{eq:basic}, cf.~\eqref{def:Q_T}, with a general $G$ which satisfies a list of conditions. We will show in Subsection~\ref{subsec:checkQ1-Q5} below, that $G$ given by \eqref{eq:defofmapG} fulfills these conditions, that will imply the items \ref{eq:QBtheta_subset_Btheta}--\ref{prop:Q_cont} of the following statement. We define the tube
\begin{equation}\label{eq:tube}
      E^+_\theta:=\{u\in E\mid 0\leq u\leq\theta\}.
\end{equation}
For the case $d=1$, we recall also that $\M$ denotes the set of all decreasing and right-continuous functions $f:\R\to[0,\theta]$, cf.~Remark~\ref{rem:inclus} below. 
\begin{theorem}[{{cf. \cite[Proposition 5.4]{FT2017a}}}]\label{thm:Qholds}
Let  \eqref{as:chiplus_gr_m} and \eqref{as:aplus_gr_aminus} hold. Let  $E=L^\infty(\X)$ and $(Q_t)_{t\geq0}$ be the semi-flow  \eqref{def:Q_T} on the cone $\{0\leq f\in L^\infty(\X)\}$. Then, for each $t>0$, $Q=Q_t$ satisfies the following properties:
\begin{enumerate}[label=\textnormal{(Q\arabic*)}]
    \item $Q$ maps each of sets $E^+_\theta$, $E^+_\theta\cap C_b(\X)$, $E^+_\theta\cap\Buc$ into itself; \label{eq:QBtheta_subset_Btheta}
    \item let $T_y$, $y\in\X$, be a translation operator, given by  \label{prop:QTy=TyQ}
      \begin{equation}\label{shiftoper}
        (T_y f)(x)=f(x-y), \quad x\in\X,
      \end{equation}
      then
      \begin{equation}
        (Q T_{y}f)(x)=(T_{y}Qf)(x), \quad x,y\in\X,\ f\in E^+_\theta;\label{eq:QTy=TyQ}
      \end{equation}
    \item $Q0=0$, $Q\theta=\theta$, and $Q r>r$, for any constant $r\in(0,\theta)$; \label{prop:Ql_gr_l}
    \item if $f,g\in E^+_\theta$, $f \leq g $, then $Qf \leq Qg$; \label{prop:Q_preserves_order}
    \item if $f_n,f\in E^+_\theta$, $f_{n}\locun f$, then $(Qf_{n})(x)\to (Qf)(x)$ for (a.a.) $x\in\X$; \label{prop:Q_cont}
    \item if $d=1$, then $Q:\M\to\M$.\label{prop:Q_Mtheta}
  \end{enumerate}
\end{theorem}
  Here and below $\locun$ denotes the locally uniform convergence of functions on $\X$ (in other words, $f_n\1_\La$ converge to $f\1_\La$ in $E$, for each compact $\La\subset\X$).  

The property \ref{eq:QBtheta_subset_Btheta} states that the solution $u(\cdot,t)$ remains in the tube $E_\theta^+$ for all $t>0$ if only $u(\cdot,0)$ is in this tube. In Remark~\ref{rem:necineq} below, we will show that, under \eqref{as:chiplus_gr_m}, the assumption \eqref{as:aplus_gr_aminus} is necessary to the fact that the set $E_\theta^+$ is invariant for $Q_t$, $t>0$.

The property \ref{prop:Q_preserves_order} means that the comparison principle holds for the solutions to \eqref{eq:basic}. Namely, if $u_1,u_2$ are classical solutions to \eqref{eq:basic} on $\R_+$ and $0\leq u_{1}(x,0)\leq u_{2}(x,0)\leq \theta$, $x\in\X$, then, for all $t\in\R_+$, (a.a.) $x\in\X$,
  \begin{equation}\label{eq:comparineq}
    0\leq u_{1}(x,t)\leq u_{2}(x,t) \leq \theta.
  \end{equation}
  See also Proposition~\ref{prop:fullcomp} below.

Our proof for the first part of Theorem~\ref{thm:trwexist} is based on an abstract result, for the case $d=1$, by Yagisita \cite{Yag2009} for a continuous semi-flow which satisfies \ref{prop:QTy=TyQ}--\ref{prop:Q_Mtheta} on $\M$ and has an appropriate super-solution (see Proposition~\ref{prop:trwexists} below for details). As an application, Yagisita considered a generalization of the equation \eqref{eq:basic} with a local $G$ in \eqref{eq:defofG}, i.e. with $\kn=0$ (and for $d=1$).

Early, in \cite{CDM2008}, it was shown how to reduce the study of traveling waves of the form \eqref{eq:deftrw} for the case $d>1$ to the study of the case $d=1$, cf. Proposition~\ref{prop:monot_sol} below; and,  for a continuous anisotropic kernel $a^+$ and for also a generalization of a local $G$ in \eqref{eq:defofG}, the traveling waves for \eqref{eq:basic} were studied using the technique of sub- and super-solutions; see also \cite{SLW2011}. For  generalizations in the case of local reaction depending on space variable (i.e. $\kn=0$ and  $\kl, m$ depend on $x$), see e.g. \cite{LZ2016,S-S2016}

The case of a nonlocal $G$ in \eqref{eq:basic}--\eqref{eq:defofG} appeared more difficult for analysis. The only known results for the case $\kn\neq 0$ in \eqref{eq:defofG} were obtained in \cite{YY2013,WZ2006} for the case of a symmetric quickly decaying kernel $a^+$, the latter mean that the integral in \eqref{aplusexpint1} is finite \emph{for all} $\mu>0$.

In the present paper, we will find an upper estimate for $c_*(\xi)$, see \eqref{cstarestimate} and \eqref{aplusexpla} below. 
Note that the present and forthcoming papers \cite{FKT100-2,FKT100-3} are based on our unpublished preprint \cite{FKT2015} and thesis \cite{Tka2017}. In particular, in \cite{FKT100-2}, we will prove that the estimate \eqref{cstarestimate} is, as a matter of fact, equality, namely,
\[
  c_*(\xi)=\min_{\la>0} \frac{1}{\la}\biggl(\kap \int_\X a^+(x) e^{\la x\cdot \xi}\,dx -m \biggr).
\]
(that coincides with the result in \cite{CDM2008} for $\kn=0$). We will find also in \cite{FKT100-2} the exact asymptotic of the profile $\psi$ at $\infty$, that implies, in particular, \eqref{eq:trwexpint}. Note that, the quite technical result \eqref{eq:trwexpint} is crucial for the analysis of traveling waves used in \cite{FKT100-2} which is based on the usage of the Laplace transform. 

It is worth noting also  that, in \cite{Wei1982a}, Weinberger considered spreading speeds of a discrete-time dynamical system $u_{n}=Qu_{n-1}$ constructed by a mapping $Q$ on $E=C_b(\X)$ which satisfies the properties \ref{eq:QBtheta_subset_Btheta}--\ref{prop:Q_cont}. He has also obtained results about a traveling wave solution (in discrete time), however, under an additional assumption that $Q$ is a compact mapping on $E=C_b(\X)$ in the topology of the locally uniform convergence. The traveling wave appeared the limit of a subsequence of appropriately chosen sequence $(u_n)_{n\in\N}$.  However, for the equation $\eqref{eq:basic}$, it is unclear how to check whether the operator $Q=Q_t$, given by \eqref{def:Q_T}, is compact on $E=C_b(\X)$ even for the local $G$ in \eqref{eq:defofG} ($\kn=0$); and hence we can't apply Weinberger's results. On the other hand, Yagisita in \cite{Yag2009} has pointed out that, considering traveling waves \eqref{eq:deftrw} with monotone profiles $\psi$, the existence of the limit above follows from Helly's theorem, which implies that $Q$ is compact on $\M$ in the topology of the locally uniform convergence. Note also that a modification of Weinberger's results about spreading speeds for continuous time for the equation \eqref{eq:basic} with an arbitrary $u_0\in E_\theta^+$ will be considered in \cite{FKT100-3}.

The paper is organized as follows. In Section~\ref{sec:semiflow}, we check properties \ref{eq:QBtheta_subset_Btheta}--\ref{prop:Q_cont} of Theorem~\ref{thm:Qholds}, and prove the strong maximum principle for the case $E=\Buc$ (see Theorem~\ref{thm:strongmaxprinciple}, cf. e.g. \cite{CDM2008} for $\kn=0$). 
In Section~\ref{sec:tr-waves}, we prove \ref{prop:Q_Mtheta} (see Proposition~\ref{prop:Qtilde}) and Theorem~\ref{thm:trwexist}. 

\section{Properties of semi-flow}\label{sec:semiflow}

\subsection{Verification of properties \ref{eq:QBtheta_subset_Btheta}--\ref{prop:Q_cont}}\label{subsec:checkQ1-Q5}
To prove \ref{eq:QBtheta_subset_Btheta}--\ref{prop:Q_cont}, we are going to use \cite[Proposition 5.4]{FT2017a}. To this end, we will check the assumptions of the latter statement. Let the mapping $G$ be given by \eqref{eq:defofmapG}. Then, under \eqref{as:chiplus_gr_m}, we have, by \eqref{theta_def},
\begin{equation}\label{eq:oldA2}
    0=G0\leq Gv \leq G\theta=\kap -m, \quad v\in E_\theta^+,
\end{equation}
cf. \eqref{eq:tube}. Moreover, it is easy to see that
\begin{equation}\label{eq:oldA8}
    Gr<\kap -m, \quad r\in(0,\theta).
\end{equation}
Note also, that, for $T_y$, $t\in\X$ given by \eqref{shiftoper}, we evidently have
\begin{equation}\label{eq:oldA7}
    (T_yGv)(x)=(GT_yv)(x), \quad x\in\X, \ v\in E_\theta^+.
\end{equation}

We denote also by
\begin{equation}\label{eq:defH}
    Hu:=\kap a^+*u-mu-u Gu
\end{equation}
the right hand side of the equation \eqref{eq:basic}. 

Let \eqref{as:aplus_gr_aminus} hold.
Then, for $u,v\in E_\theta^+$ with $u\leq v$, we have, by \eqref{theta_def}, \eqref{eq:defofmapG}, that $0\leq Gv\leq \kap -m$ and $Gv-Gu=\kl(v-u)+\kn a^-*(v-u)$, and hence
\begin{align*}
Hv-Hu&=\kap a^+*(v-u)-m(v-u)-(v -u)Gv-u(Gv-Gu)
\\&\geq J_\theta*(v-u)-(\kap +\theta \kl)(v -u). 
\end{align*}
Therefore, there exists $p=\kap +\theta \kl>0$, such that the operator $H$ is quasi-monotone on $E^+_\theta$, namely, 
\begin{equation}\label{eq:HisQuasiMon}
    Hu+pu\leq Hv+pv, \qquad u,v\in E_\theta^+, \ u\leq v.
\end{equation}

We will use also the following simple lemmas in the sequel.

\begin{lemma}\label{le:simple}
  Let $a\in L^1(\X)$, $f\in E$. Then $a*f\in\Buc$. Moreover, if $v\in C_b(I\to E)$, $I\subset\R_+$, then 
  $a*v\in C_b(I\to \Buc)$.
\end{lemma}
\begin{proof}
  The convolution is a bounded function, as
  \begin{equation}\label{convbdd}
    |(a*f)(x)|\leq \|f\|_E \,\|a\|_{L^1(\X)}, \qquad a\in L^1(\X), f\in E.
  \end{equation}
  Next, let $a_n\in C_0(\X)$, $n\in\N$, be such that $\|a-a_n\|_{L^1(\X)}\to0$, $n\to\infty$. For any $n\geq1$, the proof of that $a_n*f\in\Buc$ is straightforward. Next, by \eqref{convbdd}, $\|a*f-a_n*f\|_E\to0$, $n\to\infty$. Hence $a*u$ is a uniform limit of uniformly continuous functions that fulfills the proof of the first statement. The second statement is followed from the first one and the inequality \eqref{convbdd}.
\end{proof}

 \begin{lemma}\label{convluc}
Let $a\in L^1(\X)$, $\{f_n,f\}\subset L^\infty(\X)$, $\|f_n\|\leq C$, for some $C>0$, and $f_n\locun f$. 
Then $a*f_n \locun a*f$.
\end{lemma}
\begin{proof}
  Let $\{a_m\}\subset C_0(\X)$ be such that $\|a_m-a\|_{L^1(\X)}\to0$, $m\to\infty$, and denote $A_m:=\mathrm{supp}\, a_m$. Note that, there exists $D>0$, such that $\|a_m\|_{L^1(\X)}\leq D$, $m\in\N$. Next, for any compact $\La\subset\X$,
  \begin{align}
  |\1_\La (x) (a_m*(f_n-f))(x)|&\leq \int_\X \1_{A_m}(y) \1_\La (x) |a_m(y)| |f_n(x-y)-f(x-y)|\,dy\notag\\
  &\leq \|a_m\|_{L^1(\X)} \|\1_{\La_m}(f_n-f)\|\to0, n\to\infty, \label{eq:justappeared}
  \end{align}
  for some compact $\La_m\subset\X$. Next,
  \begin{align*}
  \|\1_\La (a*(f_n-f))\|&\leq\|\1_\La (a_m*(f_n-f))\|+\|\1_\La ((a-a_m)*(f_n-f))\|
  \\&\leq D\|\1_{\La_m}(f_n-f)\|+(C+\|f\|)\|a-a_m\|_{L^1(\X)},
  \end{align*}
  and the second term may be arbitrary small by a choice of $m$.
\end{proof}

\begin{remark}\label{rem:contofconvaa}
By the first inequality in \eqref{eq:justappeared} and the dominated convergence theorem, we can conclude that $f_n(x)\to f(x)$ a.e. implies that $(a*f_n)(x)\to (a*f)(x)$ a.e.
\end{remark}

By Lemma~\ref{convluc}, both operators $Av=\kap  a^+*v $ and $Gv=\kl v+\kn a^-*v$ are continuous in the topology of the locally uniformly convergence. 

Because of \eqref{eq:oldA2}, \eqref{eq:oldA8}, \eqref{eq:oldA7},  \eqref{eq:HisQuasiMon}, and the continuity of $G$ in both uniform and locally uniform convergences inside the tube $E_\theta^+$, one can apply \cite[Proposition 5.4]{FT2017a} to get the properties \ref{eq:QBtheta_subset_Btheta}--\ref{prop:Q_cont} of Theorem~\ref{thm:Qholds}. 

\begin{remark}
We assumed in \cite{FT2017a} also that the condition \eqref{as:a+nodeg} below holds, however, it is straightforward to check that this was not used to prove \cite[Proposition 5.4]{FT2017a}. 
\end{remark} 

\subsection{Around the comparison principle}
For each $0\leq T_1<T_2<\infty$, let $\x_{T_1,T_2}$ denote the Banach space of all continuous mappings from $[T_1,T_2]$ to $E$ with the norm 
\[
  \|u\|_{T_1,T_2}:=\sup_{t\in[T_1,T_2]}\|u(\cdot,t)\|_E.
\]
For any $T>0$, we set also $\x_T:=\x_{0,T}$ and consider the subset $\Y_T\subset\x_T$ of all mappings which are continuously differentiable on $(0,T]$. Here and below, we consider the left derivative at $t=T$ only. We consider also the vector space $\x_\infty$ of all continuous mappings from $\R_+$ to $E$.

Note that, by \eqref{eq:HisQuasiMon}, we can apply \cite[Theorem~2.3]{FT2017a} to get the following statement, which is nothing but the combination of \ref{eq:QBtheta_subset_Btheta} and \ref{prop:Q_preserves_order}.
\begin{proposition}\label{prop:compar}
Let \eqref{as:chiplus_gr_m} and \eqref{as:aplus_gr_aminus} hold. Let functions $u_1,u_2$ be classical solutions to \eqref{eq:basic} on $\R_+$ with the corresponding initial conditions which satisfy $0\leq u_{1}(x,0)\leq u_{2}(x,0)\leq \theta$  for (a.a.) $x\in\X$. Then \eqref{eq:comparineq} holds. In particular, $0\leq u(\cdot,0)\leq \theta$ for (a.a.) $x\in\X$ implies that $0\leq u(x,t)\leq\theta$ for $t>0$ and (a.a.) $x\in\X$. 
\end{proposition}

\begin{remark}\label{rem:necineq}
The condition \eqref{as:aplus_gr_aminus} is a necessary one for Proposition~\ref{prop:compar}. Indeed, let the condition \eqref{as:aplus_gr_aminus} fail in a ball $B_{r}(y_0)$ only, ${r}>0$, $y_0\in\X$, i.e. $J_\theta(x)<0$, for a.a. $x\in B_{r}(y_0)$, where $J_\theta$ is given by \eqref{diffofkernels}. Take any $y\in B_{r}(y_0)$ with $\frac{{r}}{4}<|y-y_0|<\frac{3{r}}{4}$, then $y_0\notin B_{\frac{{r}}{4}}(y)$ whereas $B_{\frac{{r}}{4}}(y)\subset B_{r}(y_0)$.
Take $u_0\in\Buc$ such that $u_0(x)=\theta$, $x\in \X\setminus B_{\frac{{r}}{4}}(y_0-y)$, and $u_0(x)< \theta$, $x\in B_{\frac{{r}}{4}}(y_0-y)$. Since $\int_\X J_\theta(x)\,dx=\kap -\kn\theta=m+\kl\theta$, one has
\begin{align*}
\frac{\partial u}{\partial t}(y_0,0)&=-(m+\kl\theta)\theta+\kap (a^+*u)(y_0,0)-\kn\theta (a^-*u)(y_0,0)\notag\\&=(J_\theta*u)(y_0,0)-(\kap -\kn\theta)\theta=(J_\theta*(u_0-\theta))(y_0)
\\&
=\int_{B_{\frac{{r}}{4}}(y)} J_\theta(x)(u_0(y_0-x)-\theta)\,dx>0,\notag
\end{align*}
Therefore, $u(y_0,t)>u(y_0,0)=\theta$, for small enough $t>0$, and hence, the statement of Proposition~\ref{prop:compar} does not hold in this case. 
\end{remark}

As a simple corollary of \ref{eq:QBtheta_subset_Btheta}--\ref{prop:Q_cont}, we will show that the semi-flow $(Q_t)_{t\geq0}$ preserves functions which are monotone along a given direction. More precisely, 
a function $f\in L^\infty(\X)$ is said to be increasing (decreasing, constant) along the vector $\xi\in S^{d-1} $ (recall that $ S^{d-1} $ denotes a unit sphere in $\X$ centered at the origin) if, for a.a. $x\in\X $, the function 
$f(x+s\xi)=(T_{-s\xi}f)(x)$ is increasing (decreasing, constant) in $s\in\R$, respectively.

\begin{proposition}\label{prop:monot_along_vector_sol}
Let \eqref{as:chiplus_gr_m} and \eqref{as:aplus_gr_aminus} hold. Let $u_0\in E^+_\theta$ be the initial condition for the equation \eqref{eq:basic} which is 
increasing (decreasing, constant) along a vector $\xi\in S^{d-1}  $; and $u(\cdot,t)\in E^+_\theta$, $t\geq0$, be the corresponding solution (cf. Proposition~\ref{prop:compar}). Then, for any $t>0$, $u(\cdot,t)$ is increasing (decreasing, constant, respectively) along the $\xi$.
\end{proposition}
\begin{proof}
Let $u_0$ be decreasing along a $\xi\in S^{d-1}  $. Take any $s_1\leq s_2$ and consider
two initial conditions to \eqref{eq:basic}: $u_0^i(x)=u_0(x+s_i\xi)=(T_{-s_i\xi}u_0)(x)$, $i=1,2$. Since $u_0$ is decreasing, $u_0^1(x)\geq u_0^2(x)$, $x\in\X$. Then, by Theorem~\ref{thm:Qholds},
\[
T_{-s_1\xi}Q_tu_0=Q_tT_{-s_1\xi}u_0=Q_t u_0^1\geq Q_tu_0^2=Q_tT_{-s_2\xi}u_0=
T_{-s_2\xi}Q_tu_0,
\]
that proves the statement. The cases of a decreasing $u_0$ can be considered in the same way. The constant function along a vector is decreasing and decreasing simultaneously.
\end{proof}

For each $T>0$ and $u\in\Y_T$, one can define
\begin{equation}\label{Foper}
  (\mathcal{F}u)(x,t):=\dfrac{\partial u}{\partial t}(x,t)-\kap (a^{+}*u)(x,t)+mu(x,t)+u(x,t) \bigl( Gu \bigl)(x,t)
\end{equation}
for all $t\in(0,T]$ and $x\in\X$ (a.a. $x\in\X$ in the case $E=L^\infty(\X)$).

By \cite[Theorem~2.3]{FT2017a}, we will also get the following counterpart of Proposition~\ref{prop:compar}.
  \begin{proposition}\label{prop:fullcomp}
  Let \eqref{as:chiplus_gr_m} and \eqref{as:aplus_gr_aminus} hold. Let $T>0$ be fixed and $u_1,u_2\in\Y_T$ be such that, for all $t\in(0,T]$, $x\in\X$,
    \begin{gather}
    (\mathcal{F}u_1)(x,t)\leq (\mathcal{F}u_2)(x,t),\label{eq:max_pr_BUC:ineq}\\
    0 \leq u_1(x,t)\leq\theta, \qquad 0 \leq u_2(x,t)\leq \theta,\notag\\
      0\leq u_{1}(x,0)\leq u_{2}(x,0)\leq \theta.\notag
    \end{gather}
    Then \eqref{eq:comparineq} holds for all $t\in[0,T]$, $x\in\X$.
  \end{proposition}

Below, for technical reasons, we will need to extend the result of Proposition~\ref{prop:fullcomp} for a wider class of functions in the case $E=\Buc$. Namely, the expression \eqref{Foper} is well-defined for a.a. $t$ if the function $u$ is absolutely continuous in $t$ only. In view of this, for any $T\in(0,\infty]$, we define the set $\mathscr{D}_T$ of all functions $u:\X\times\R_+\to\R$, such that, for all $t\in[0,T)$, $u(\cdot,t)\in \Buc$, and, for all $x\in\X$, the function $f(x,t)$ is absolutely continuous in $t$ on $[0,T)$. Then, for any $u\in\mathscr{D}_T$, one can define the function \eqref{Foper}, for all $x\in\X$ and a.a. $t\in[0,T)$.  

\begin{proposition}\label{compprabscont}
The statement of Proposition~\ref{prop:fullcomp} remains true, if we assume that $u_1,u_2\in\mathscr{D}_T$ and, for any $x\in\X$, the inequality \eqref{eq:max_pr_BUC:ineq} holds for a.a.~$t\in(0,T)$~only.
\end{proposition}
\begin{proof}
  One can literally follow the proof of \cite[Theorem~4.2]{FT2017a}: the auxiliary function $v(x,t):=e^{Kt}(u_{2}(x,t)-u_{1}(x,t))$ with large enough $K>0$ will satisfy a proper differential equation $\frac{d}{dt}v(x,t)=\Theta(t,v(x,t))$, see \cite[(4.12)]{FT2017a}, for a.a. $t\in[0,T]$. However, the corresponding integral equation $v(x,t)=v(x,0)+\int_0^t \Theta(s,v(x,s))\,ds$ holds still for all $t\in[0,T]$, since $v$ is continuous in $t$. Hence, the rest of the proof remains the same.
\end{proof}
We are going to show now that any solution to \eqref{eq:basic} is bounded from below by a solution to the corresponding equation with `truncated' kernels~$a^\pm$.
Namely, suppose that the conditions \eqref{as:chiplus_gr_m}, \eqref{as:aplus_gr_aminus} hold. Consider a family of Borel sets $\{\Delta_R\mid R>0\}$, such that $\Delta_R\nearrow\X$, $R\to\infty$. Define, for any $R>0$, the following kernels:
\begin{equation}\label{trkern}
  a_R^{\pm}(x)=\1_{\Delta_R}(x)a^{\pm}(x),\quad x\in\X,
\end{equation}
and the corresponding `truncated' equation, cf. \eqref{eq:basic},
\begin{equation}
  \begin{cases}
    \begin{aligned}
      \dfrac{\partial w}{\partial t}(x,t)&= \kap (a_R^+*w)(x,t)-mw(x,t) - \kl w^2(x,t) \\ 
                                          &\quad -\kn w(x,t)(a_R^-*w)(x,t), \qquad x\in\X,\  t>0,
    \end{aligned}\\
    w(x,0)=w_{0}(x), \qquad \qquad \qquad \qquad \qquad \quad \, x\in\X.
  \end{cases}\label{eq:basic_R}
\end{equation}
We set
\begin{equation}\label{ARdef}
  A_R^\pm:=\int_{\Delta_R}a^\pm(x)\,dx \nearrow 1, \quad R\to\infty,
\end{equation}
by \eqref{normed}. Then the non-zero constant solution to \eqref{eq:basic_R} is equal to
\begin{equation}\label{defofthetaR}
  \theta_R=\dfrac{\kap A_R^+-m}{\kn A_R^- + \kl}\to \theta, \quad R\to\infty,
\end{equation}
however, the convergence $\theta_R$ to $\theta$ is, in general, not monotonic. Clearly, by~\eqref{as:chiplus_gr_m}, $\theta_R>0$ if only
\begin{equation}\label{bigR}
  A_R^+>\frac{m}{\kap }\in(0,1).
\end{equation}

\begin{proposition}\label{lowestsuppCub}
Let \eqref{as:chiplus_gr_m} and \eqref{as:aplus_gr_aminus} hold, and let $R>0$ be such that \eqref{bigR} holds, cf.~\eqref{ARdef}. Let $w_0\in E$ be such that $0\leq w_0\leq \theta_R,\ x\in\X$. Then there exists the unique solution $w\in\Xinf$ to \eqref{eq:basic_R}, such that
\begin{equation}\label{wlessthetaR}
0\leq w(x,t)\leq\theta_R, \quad x\in\X,\ t>0.
\end{equation}
Let $u_0\in E_\theta^+$ and $u\in\Xinf$ be the corresponding solution to \eqref{eq:basic}. If $w_0(x)\leq u_0(x), x\in \X$, then
\begin{equation}\label{ineqtrunc}
w(x,t)\leq u(x,t),\quad x\in\X, \ t>0.
\end{equation}
\end{proposition}
\begin{proof}
Denote $\Delta_R^c:=\X\setminus \Delta_R$. We have
\begin{align*}
  \theta-\theta_R &= \frac{\kn \theta A^-_R + \kl \theta - \kap A^+_R+m}{\kam (\kn A^-_R+\kl)}
  =\frac{\kap (1-A^+_R)-\kn \theta (1-A^-_R)}{\kam  (\kn A^-_R+\kl)}\\
  &=\frac{1}{\kam  (\kn A^-_R + \kl)}\int_{\Delta_R^c}\bigl( \kap  a^+(x)-\kn \theta a^-(x)\bigr)\,dx\geq0,
\end{align*}
by \eqref{as:aplus_gr_aminus}. Therefore,
\begin{equation}\label{thetaRlesstheta}
  0<\theta_R\leq\theta.
\end{equation}
Clearly, \eqref{as:aplus_gr_aminus} and \eqref{thetaRlesstheta} yield
\begin{equation}\label{as:aplus_geq_aminus_R}
\kap a_R^+(x)\geq \theta_R\kam a_R^-(x),\quad x\in\X.
\end{equation}
Thus one can apply Proposition~\ref{prop:compar} to the equation~\eqref{eq:basic_R} using trivial equalities $a^\pm_R(x)=A_R^\pm \tilde{a}^\pm_R(x)$, where the kernels $\tilde{a}^\pm_R(x)=(A_R^\pm)^{-1}a^\pm_R(x)$ are normalized, cf. \eqref{normed}; and the inequality \eqref{as:aplus_geq_aminus_R} is the corresponding analog of \eqref{as:aplus_gr_aminus}, according to \eqref{defofthetaR}. This proves the existence and uniqueness of the solution to \eqref{eq:basic_R} and the bound \eqref{wlessthetaR}.

Next, for $\mathcal{F}$ given by \eqref{Foper}, one gets from \eqref{trkern} and \eqref{eq:basic_R}, that the solution $w$ to \eqref{eq:basic_R} satisfies the following equality
\begin{align}
  (\mathcal{F}w)(x,t)=&-\kap \int_{\Delta_R^c} a^+(y) w(x-y,t)\,dy\notag
\\&+\kn w(x,t)\int_{\Delta_R^c} a^-(y)w(x-y,t)\,dy.\label{dopR}
\end{align}
By \eqref{wlessthetaR}, \eqref{thetaRlesstheta}, \eqref{as:aplus_gr_aminus}, one gets from \eqref{dopR} that
\begin{align*}
(\mathcal{F}w)(x,t)
&\leq-\kap \int_{\Delta_R^c} a^+(y) w(x-y,t)\,dy + \kn\theta\int_{\Delta_R^c} a^-(y)w(x-y,t)\,dy \\
&\leq0=(\mathcal{F}u)(x,t),
\end{align*}
 where $u$ is the solution to \eqref{eq:basic}. Therefore, we may apply Proposition~\ref{prop:fullcomp} to get the statement.
\end{proof}

In the following two propositions we consider results about stability of stationary solutions to \eqref{eq:basic}.

According to the proof of \cite[Theorems 2.2, 3.4]{FT2017a}, which implies Theorem~\ref{thm:existuniq}, the solution $u(x,t)$ to \eqref{eq:basic} may be obtained on an arbitrary time interval $[0,T]$ as follows. There exist $m\in\N$ and $0=:\tau_0<\tau_1<\ldots<\tau_m$ with $\tau_m\geq T$, such that for each $[\tau, \widehat{\tau}]:=[\tau_{k-1},\tau_k]$, $1\leq k\leq m$, there exists $r_k>0$, such that, for any $v\in \x_{\tau,\widehat{\tau}}$ with $0\leq v\leq r_k$,  $u=\lim\limits_{n\to\infty}\Phi_\tau ^n v$  in $\x_{\tau,\widehat{\tau}}$, where  
\begin{align}
&\begin{aligned}
(\Phi_\tau v)(x,t)&:=(Bv)(x,\tau,t)u_\tau(x)\\ &\quad +\int_\tau^t(Bv)(x,s,t)\kap (a^{+}*v)(x,s)\,ds,
\label{eq:exist_uniq_BUC:Phi_v}
\end{aligned}\\
&(Bv)(x,s,t):=\exp\biggl(-\int _{s}^t\bigl(m+(Gv)(x,p)\bigr)\,dp\biggr),\label{eq:exist_uniq_BUC:B}
\end{align}
for $x\in\X$, $t,s\in[\tau,T]$, and $G$ is given by \eqref{eq:defofG}.
By the uniqueness arguments, we will immediately get the following proposition.
\begin{proposition}\label{prop:startwithconst}
  Let $t_0\geq0$ be such that the solution $u$ to \eqref{eq:basic} is a constant in space at the moment of time $t_0$, namely, $u(x,t_0)\equiv u(t_0)\geq0$, $x\in\X$. Then
  this solution will be a constant in space for all further moments of time. In particular, if \eqref{as:chiplus_gr_m} holds (and hence $\beta=\kap -m>0$), then
\begin{equation}\label{homogensol}
  u(x,t)= u(t)=\frac{\theta u(t_0)}{u(t_0) (1-e^{-\beta t}) +\theta e^{-\beta t}}\geq0, \qquad x\in\X, \ t\geq t_0,
\end{equation}
and $u(t)\to \theta$, $t\to\infty$.
\end{proposition}

\begin{remark}
  Note that \eqref{homogensol} solves the classical logistic equation 
  \begin{equation}\label{eq:homogen}
    \frac{d}{dt} u(t)=\kam  u(t) (\theta - u(t)), \quad t>t_0,\quad u(t_0)\geq0.
  \end{equation}
\end{remark}

We are going to study stability of constant stationary solutions  to \eqref{eq:basic}.

\begin{proposition}\label{prop:statsol}
Let \eqref{as:chiplus_gr_m} and \eqref{as:aplus_gr_aminus} hold. Then $u^*\equiv \theta$ is a uniformly and asymptoticaly stable solution to \eqref{eq:basic}, whereas $u_*\equiv 0$ is an unstable solution to~\eqref{eq:basic}.
\end{proposition}
\begin{proof}
Let $H$ and $J_\theta$ be given by \eqref{eq:defH} and \eqref{diffofkernels}, correspondingly. 
Find the linear operator $H'(u)$ on $E$: namely, for $v\in E$,
\begin{equation}\label{derofG}
  H'(u)v=\kap (a^+*v)-mv-\kn v(a^-*u)-\kn u(a^-*v)-2\kl uv.
\end{equation}
Therefore, by \eqref{diffofkernels}, 
\[
  H'(\theta)v =J_\theta*v-(\kap +\kl\theta) v.
\]
By \eqref{diffofkernels}, $\int_\X J_\theta(x)\,dx = \kap{-}\kn\theta$, thus, the spectrum $\sigma(A)$ of the operator $Av:=J_\theta *v$ on $\Buc$ is a subset of $\{|z|\leq \kap{-}\kn\theta\}\subset \mathbb{C}$. Therefore,
\[
  \sigma(H'(\theta)) \subset \bigl\{z\in\mathbb{C}\bigm\vert |z+\kap +\kl\theta|\leq \kap{-}\kn\theta\bigr\}.
\]
Therefore, $\sigma(H'(\theta)) \subset \{z\in\mathbb{C}\mid \mathrm{Re}\, z<0\}$. Hence, by e.g. \cite[Chapter VII]{DK1974}, $u^*\equiv\theta$ is uniformly and asymptotically stable solution in the sense of Lyapunov. 

Next, by \eqref{derofG},
$H'(0)v=\kap (a^+*v)-mv$. If \eqref{as:chiplus_gr_m} holds, then the operator $H'(0)$ has an eigenvalue $\kap -m>0$ whose corresponding eigenfunctions will be constants on $\X$. Therefore $\sigma(H'(0))$ has points in the right half-plane and since $H''(0)$ exists, one has, again by \cite[Chapter VII]{DK1974}, that $u_*\equiv0$ is unstable.
\end{proof}

\subsection{Strong maximum principle}
Now we are going to study the maximum principle for solutions to \eqref{eq:basic} in the space $E=\Buc$. 
For this case, we denote $\Utheta:=E_\theta^+$.

We introduce also the following assumption: 
there exist $\rho,\delta>0$ such that, cf. \eqref{diffofkernels},
\begin{equation}\label{as:aplus-aminus-is-pos}\tag{\ref{as:aplus-aminus-is-pos1d}${}'$}
    J_\theta(x)=\kap a^{+}(x)-\kn \theta a^{-}(x)\geq\rho \text{ \ for a.a. }  |x|\leq\delta.
\end{equation}
Clearly, \eqref{as:aplus-aminus-is-pos} implies \eqref{as:aplus-aminus-is-pos1d} and implies also that the following condition holds: there exist $\rho,\delta>0$, such that
\begin{equation}\label{as:a+nodeg}\tag{\ref{as:a+nodeg1d}${}'$}
a^{+}(x)\geq\rho \text{ \ for a.a. } |x|\leq \delta.
\end{equation}

It is straightforward to check that, under assumptions \eqref{as:chiplus_gr_m}, \eqref{as:aplus_gr_aminus}, \eqref{as:a+nodeg}, one can apply \cite[Proposition 5.2]{FT2017a}, that yields the following statement about strict positivity of solutions to \eqref{eq:basic}.
\begin{proposition}\label{prop:u_gr_0}
Let $E=\Buc$ and \eqref{as:chiplus_gr_m}, \eqref{as:aplus_gr_aminus}, \eqref{as:a+nodeg}
hold. Let $u_0\in \Utheta$, $u_0\not\equiv0$, $u_0\not\equiv\theta$, be~the~initial condition to \eqref{eq:basic}, and $u\in\Xinf$ be the corresponding solution. Then
\[
u(x,t)>\inf_{\substack{y\in\X\\ s>0}}u(y,s)\geq0, \qquad x\in\X, t>0.
\]
\end{proposition}

In contrast to the case of the infimum, the solution to \eqref{eq:basic} may attain its supremum but not the value $\theta$. As a matter of fact, under \eqref{as:aplus-aminus-is-pos}, a much stronger statement than unattainability of $\theta$ does hold.

\begin{theorem}\label{thm:strongmaxprinciple}
Let  $E=\Buc$ and \eqref{as:chiplus_gr_m}, \eqref{as:aplus_gr_aminus}, \eqref{as:aplus-aminus-is-pos}
hold. Let $u_1,u_2\in\Xinf$ be two solutions to \eqref{eq:basic}, such that $0\leq u_1(x,t)\leq u_2(x,t)\leq\theta$, $x\in\X$, $t\geq0$. Then either $u_1(x,t)= u_2(x,t)$, $x\in\X$, $t\geq0$ or
$u_1(x,t)< u_2(x,t)$, $x\in\X$, $t>0$.
\end{theorem}
\begin{proof}
Let $u_1(x,t)\leq u_2(x,t)$, $x\in\X$, $t\geq0$, and suppose that there exist $t_0>0$, $x_0\in\X $,
such that $u_1(x_0,t_0)=u_2(x_0,t_0)$. Define $w:=u_2-u_1\in\Xinf$. Then $w(x,t)\geq0$ and $w(x_0,t_0)=0$, hence $\frac{\partial}{\partial t}w(x_0,t_0)=0$. Since both $u_1$ and $u_2$ solve \eqref{eq:basic}, one easily gets that $w$ satisfies the following linear equation
\begin{multline}\label{lineareqw}
  \frac{\partial}{\partial t}w(x,t) = (J_\theta* w)(x,t) + \kn(\theta-u_1(x,t))(a^-*w)(x,t)\\
	                                     -w(x,t)\bigl(\kl \bigl(u_2(x,t)+u_1(x,t)\bigl) + \kn(a^-*u_2)(x,t) +m \bigr);
\end{multline}
or, at the point $(x_0,t_0)$, we will have
\begin{equation}\label{atx0t0}
  0 = (J_\theta* w)(x_0,t_0) + \kn(\theta-u_1(x_0,t_0))(a^-*w)(x_0,t_0).
\end{equation}
Since the both summands in \eqref{atx0t0} are nonnegative, one has $(J_\theta* w)(x_0,t_0)=0$. Then, by \eqref{as:aplus-aminus-is-pos}, we have that $w(x,t_0)=0$, for all $x\in B_\delta (x_0)$. Using the same arguments as in the proof of \cite[Proposition 5.2]{FT2017a},, one gets that $w(x,t_0)=0$, $x\in\X$. Then, by Proposition~\ref{prop:startwithconst}, $w(x,t)=0$, $x\in\X$, $t\geq t_0$. Finally, one can reverse the time in the linear equation \eqref{lineareqw} (cf.~the proof of \cite[Proposition 5.2]{FT2017a}), and the uniqueness arguments imply that $w\equiv 0$, i.e. $u_1(x,t)= u_2(x,t)$, $x\in\X$, $t\geq0$. The statement is proved.
\end{proof}

By choosing $u_2\equiv\theta$ in Theorem~\ref{thm:strongmaxprinciple}, we immediately get the following
\begin{corollary}\label{cor:lesstheta}
Let  $E=\Buc$ and \eqref{as:chiplus_gr_m}, \eqref{as:aplus_gr_aminus}, \eqref{as:aplus-aminus-is-pos}
hold. Let $u_0\in \Utheta$, $u_0\not\equiv\theta$, be~the~initial condition to \eqref{eq:basic}, and $u\in\Xinf$ be the corresponding solution. Then
$u(x,t)<\theta$, $x\in\X$, $t>0$.
\end{corollary}

\section{Traveling waves}\label{sec:tr-waves}

Through this section, $E=L^\infty(\X)$. Similarly to the above, we denote by $\Y_\infty$ the subset of $\x_\infty$ of all continuously differentiable mappings from $(0,\infty)$ to $E$.

Recall that $\M$ denotes the set of all decreasing and right-continuous functions $f:\R\to[0,\theta]$.

\begin{remark}\label{rem:inclus}
  There is a natural embedding of $\M$ into $L^\infty(\R)$. According to this, for a function $f\in L^\infty(\R)$, the inclusion $f\in\M$ means that there exists $g\in\M$, such that $f=g$ a.s. on $\R$.
\end{remark}

Recall also the definition of a traveling wave solution.
\begin{definition}\label{def:trw}
A function $u\in \Y_\infty$ is said to be a traveling
wave solution to \eqref{eq:basic} with
a speed $c\in\R$ and in a direction $\xi\in S^{d-1}  $ if  there
exists a profile $\psi\in\M$, such that \eqref{eq:deftrw} holds. 
\end{definition}

We will use some ideas and results from \cite{Yag2009}.

To study traveling wave solutions to \eqref{eq:basic}, it is natural to consider
the corresponding initial conditions of the form
 \begin{equation}\label{trwvincond}
u_0(x)=\psi(x\cdot\xi),
\end{equation}
for some $\xi\in S^{d-1}  $, $\psi\in\M$. Then the solutions will have a special form as well, namely, the following proposition holds.
\begin{proposition}\label{prop:monot_sol}
Let $\xi\in S^{d-1}  $, $\psi\in\M$, and an initial condition to \eqref{eq:basic} be given by
$u_0(x)=\psi(x\cdot\xi)$, a.a.\;$x\in\X$; let also $u\in\x_\infty$ be the corresponding solution. Then there exist a function $\phi:\R\times\R_+\to[0,\theta]$, such that $\phi(\cdot,t)\in\M$, for any $t\geq0$, and
\begin{equation}\label{repres}
  u(x,t)=\phi(x\cdot\xi,t),\quad t\geq0,\ \mathrm{a.a.}\ x\in\X.
\end{equation}

Moreover, there exist functions $\widecheck{a}^\pm$ (depending on $\xi$) on $\R$ with
$0\leq \widecheck{a}^\pm\in L^1(\R)$, $\int_\R \widecheck{a}^\pm(s)\,ds=1$, such that $\phi$ is a solution to the following one-dimensional version of \eqref{eq:basic}:
\begin{equation}
  \begin{cases}
  	\begin{aligned}
			\dfrac{\partial \phi}{\partial t}(s,t)&=\kap (\widecheck{a}^{+}*\phi)(s,t)-m\phi(s,t) -\kl \phi^2(s,t) \\&\quad
				-\kn\phi(s,t)(\widecheck{a}^{-}*\phi)(s,t), \qquad t>0, \ \mathrm{a.a.}\ s\in\R,
		\end{aligned}\\
		\phi(s,0)=\psi(s),\qquad \mathrm{a.a.}\ s\in\R.
	\end{cases}\label{eq:basic_one_dim}
\end{equation}
\end{proposition}
\begin{proof}
Choose any $\eta\in S^{d-1}  $ which is orthogonal to the $\xi$. Then the initial condition $u_0$ is constant along $\eta$, indeed, for any $s\in\R$,
\[
	u_0(x+s\eta)=\psi((x+s\eta)\cdot\xi)=\psi(x\cdot\xi)=u_0(x),\quad \mathrm{a.a.}\ x\in\X.
\]
Then, by Proposition~\ref{prop:monot_along_vector_sol}, for any fixed $t>0$, the solution $u(\cdot,t)$ is constant along $\eta$ as well. Next, for any $\tau\in\R$, there exists $x\in\X$ such that $x\cdot\xi=\tau$; and, clearly, if $y\cdot\xi=\tau$ then $y=x+s\eta$, for some $s\in\R$ and some $\eta$ as above. Therefore, if we just set, for a.a. $x\in\X$, $\phi(\tau,t):=u(x,t)$, $t\geq0$, this definition will be correct a.e.\! in $\tau\in\R$; and it will give \eqref{repres}. Next, for a.a. fixed $x\in\X$, $u_0(x+s\xi)=\psi(x\cdot\xi+s)$ is decreasing in $s$, therefore, $u_0$ is decreasing along the $\xi$, and by Proposition~\ref{prop:monot_along_vector_sol},
$u(\cdot,t)$, $t\geq0$, will be decreasing along the $\xi$ as well. The latter means that, for any $s_1\leq s_2$, we have, by \eqref{repres},
\[
	\phi(x\cdot\xi+s_1,t)=u(x+s_1\xi,t)\geq u(x+s_2\xi,t)=\phi(x\cdot\xi+s_2,t),
\]
and one can choose in the previous any $x$ which is orthogonal to $\xi$ to prove that $\phi$ is decreasing in the first coordinate.

To prove the second statement, for $d\geq2$, choose any $\{\eta_{1},\ \eta_{2},\ ...,\ \eta_{d-1}\}\subset S^{d-1}  $ which form a complement of $\xi\in S^{d-1}  $ to an orthonormal basis in $\X $.
Then, for a.a.\,\,$x\in\X$, with $x=\sum_{j=1}^{d-1}\tau_j\eta_j+s\xi$, $\tau_1,\ldots,\tau_{d-1},s\in\R$, we have (using an analogous expansion of $y$ inside the integral below an taking into account that any linear transformation of orthonormal bases preserves volumes)
\begin{align}
&\quad (a^\pm*u)(x,t)=\int_\X a^\pm(y)u(x-y,t)dy\notag\\
&=\int_\X a^\pm\biggl(\sum_{j=1}^{d-1}\tau_j'\eta_j+s'\xi\biggr)\,
    u\biggl(\sum_{j=1}^{d-1}(\tau_j-\tau_j')\eta_j+(s-s')\xi,t\biggr)\,d\tau_{1}'\ldots d\tau_{d-1}'ds'\notag\\
&=\int_\R\Biggl(\int_{\R^{d-1}}a^\pm\biggl(\sum_{j=1}^{d-1}\tau_j'\eta_j+s'\xi\biggr)\,
d\tau_1'\ldots d\tau_{d-1}'\Biggr)u\bigl((s-s')\xi,t\bigr)\,ds',\label{reducingto1dim}
\end{align}
where we used again Proposition~\ref{prop:monot_along_vector_sol} to show that $u$ is constant along the vector $\eta=\sum_{j=1}^{d-1}(\tau_j-\tau_j')\eta_j$ which is orthogonal to the $\xi$.

Therefore, one can set
\begin{equation}\label{apm1dim}
\widecheck{a}^\pm(s):=\begin{cases}
\displaystyle \int_{\R^{d-1}} a^\pm(\tau_1\eta_1+\ldots+\tau_{d-1}\eta_{d-1}+s\xi)\,d\tau_1\ldots d\tau_{d-1}, &d\geq2,\\[3mm]
a^\pm(s\xi), &d=1.
\end{cases}
\end{equation}
It is easily seen that $\widecheck{a}^\pm=\widecheck{a}^\pm_\xi$ does not depend on the choice of $\eta_1,\ldots,\eta_{d-1}$, which constitute a basis in the space $H_\xi:=\{x\in\X\mid x\cdot\xi=0\}=\{\xi\}^\bot$.
Note that, clearly,
\begin{equation}\label{cleareq}
\int_\R \widecheck{a}^\pm(s)\,ds=\int_\X a^\pm(y)\,dy=1.
\end{equation}
Next, by \eqref{repres}, $u\bigl((s-s')\xi,t\bigr)=\phi(s-s',t)$, therefore, \eqref{reducingto1dim} may be rewritten as
\[
(a^\pm*u)(x,t)=\int_\R \widecheck{a}^\pm(s')\phi(s-s',t\bigr)\,ds'=:(\widecheck{a}^\pm*\phi)(s,t),
\]
where $s=x\cdot\xi$. The rest of the proof is obvious now.
\end{proof}
\begin{remark}\label{rem:multi-one}
Let $\xi\in S^{d-1} $ be fixed and $\widecheck{a}^\pm$ be defined by \eqref{apm1dim}. Let $\phi$ be a traveling wave solution to the equation \eqref{eq:basic_one_dim} (in the sense of Definition~\ref{def:trw}, for $d=1$) in the direction $1\in S^0=\{-1,1\}$, with a profile $\psi\in\M$ and a speed $c\in\R$. Then the function $u$ given by
\begin{equation}\label{tw1d}
u(x,t)=\psi(x\cdot\xi-ct)=\psi(s-ct)=\phi(s,t),
\end{equation}
for $x\in\X$, $t\geq0$, $s=x\cdot\xi\in\R$,
is a traveling wave solution to \eqref{eq:basic} in the direction $\xi$, with the profile $\psi$ and the speed $c$.
\end{remark}

\begin{remark}\label{incrinsteadofdecr}
  One can realize all previous considerations for increasing traveling wave, increasing solution along a vector $\xi$ etc. Indeed, it is easily seen that the function $\tilde{u}(x,t)=u(-x,t)$ with the initial condition $\tilde u_0(x)=u_0(-x)$ is a solution to the equation \eqref{eq:basic} with $a^\pm$ replaced by $\tilde{a}^\pm(x)=a^\pm(-x)$; note that $(a^\pm*u)(-x,t)=(\tilde{a}^\pm*\tilde{u})(x,t)$.
\end{remark}

\begin{remark}\label{shiftoftrw}
  It is a straightforward application of \eqref{eq:QTy=TyQ}, that if $\psi\in\M$, $c\in\R$ gets \eqref{eq:deftrw} then, for any $s\in\R$, $\psi(\cdot+s)$ is a traveling wave to \eqref{eq:basic} with the same $c$.
\end{remark}

We can prove now the following simple statement, which implies, in particular, the property \ref{prop:Q_Mtheta} in Theorem~\ref{thm:Qholds}. Consider one-dimensional equation \eqref{eq:basic_one_dim}, where $\widecheck{a}^\pm$ are given by \eqref{apm1dim}. The latter equality together with \eqref{as:aplus_gr_aminus} imply \eqref{as:aplus_gr_aminus-intro} that is equivalent to 
\begin{equation}\label{acheckpos}
  \kap \widecheck{a}^+(s)\geq \kn \theta \widecheck{a}^-(s), \quad \text{a.a.} \ s\in\R.
\end{equation} 
\begin{proposition}\label{prop:Qtilde}
Let \eqref{as:chiplus_gr_m} and  \eqref{as:aplus_gr_aminus-intro} hold, and let $\xi\in S^{d-1}  $ be fixed.
  Define, for an arbitrary $t>0$, the mapping $\widetilde{Q}_{t}:L^{\infty}(\R)\to L^{\infty}(\R)$ as follows: $\widetilde{Q}_t\psi(s)=\phi(s,t)$, $s\in\R$, where $\phi:\R\times\R_+\to[0,\theta]$ solves \eqref{eq:basic_one_dim} with $0\leq\psi\in L^{\infty}_{+}(\R)$. Then such a $\widetilde{Q}_t$ is well-defined and satisfies all properties of Theorem~\ref{thm:Qholds} (with $d=1$). 
\end{proposition}
\begin{proof}
Note that all previous results (e.g. Theorem~\ref{thm:existuniq}) hold true for the solution to \eqref{eq:basic_one_dim} as well. In particular, properties \ref{eq:QBtheta_subset_Btheta}--\ref{prop:Q_cont} of Theorem~\ref{thm:Qholds} hold true, for $Q=\widetilde{Q}_t$, $d=1$. Moreover (see the proof of \cite[Theorems 2.2, 3.4]{FT2017a} for $E=L^\infty(\X)$, which implies Theorem~\ref{thm:existuniq}), the mappings $B$ and $\Phi_\tau$, cf. \eqref{eq:exist_uniq_BUC:B}, \eqref{eq:exist_uniq_BUC:Phi_v}, map the set $\M$ into itself; as a result, we have that $\widetilde{Q}_t$ has this property as well, cf.~Remark~\ref{rem:inclus}.
\end{proof}

Now we are going to prove the existence of the traveling wave solution to \eqref{eq:basic}. Denote, for any $\la>0$, $\xi\in S^{d-1}  $,
\begin{equation}\label{aplusexpla}
  \A_\xi(\la):=\int_\X a^+(x) e^{\la x\cdot \xi}\,dx\in[0,\infty].
\end{equation}
Therefore, for a $\xi\in S^{d-1} $, the assumption \eqref{aplusexpint1} means that $\A_{\xi}(\mu)<\infty$ for some $\mu=\mu(\xi)>0$.

We will prove now the first statement of Theorem~\ref{thm:trwexist}. 

\begin{proposition}\label{prop:trwexists}
Let $\xi\in S^{d-1}  $ and the assumptions \eqref{as:chiplus_gr_m},  \eqref{as:aplus_gr_aminus-intro}, \eqref{aplusexpint1} hold.
Then there exists $c_*(\xi)\in\R$ such that
\begin{enumerate}[label={\arabic*})]
    \item for any $c\geq c_*(\xi)$, there exists a traveling wave solution, in the sense of Definition~\ref{def:trw}, with a profile $\psi\in\M$ and the speed $c$,
    \item for any $c<c_*(\xi)$, such a traveling wave does not exist.
\end{enumerate}
\end{proposition}
\begin{proof}
Let $\mu>0$ be such that \eqref{aplusexpint1} holds. Then, by \eqref{apm1dim},
\begin{align}\notag
\int_\R \widecheck{a}^+(s) e^{\mu s}ds&=\int_\R \int_{\R^{d-1}} a^\pm(\tau_1\eta_1+\ldots+\tau_{d-1}\eta_{d-1}+s\xi)e^{\mu s}\,d\tau_1\ldots d\tau_{d-1} ds\\&=\A_\xi(\mu)<\infty.\label{expintla1}
\end{align}
Clearly, the integral equality in \eqref{expintla1} holds true for any $\la\in\R$ as well, with $\A_\xi(\la)\in[0,\infty]$.

Let $\mu>0$ be such that \eqref{aplusexpint1} holds. Define a function from $\M$ by
\begin{equation}\label{defvarphi}
\varphi(s):=\theta\min\{e^{-\mu s},1\}.
\end{equation}
Let us prove that there exists $c\in\R$ such that $\bar{\phi}(s,t):=\varphi(s-ct)$ is a super-solution to \eqref{eq:basic_one_dim}, i.e.
\begin{equation}\label{supersol}
\mathcal{F}\bar{\phi}(s,t)\geq0,\quad s\in\R, t\geq0,
\end{equation}
where $\mathcal{F}$ is given by \eqref{Foper} (in the case $d=1$).
We have
\begin{align*}
  (\mathcal{F}\bar{\phi})(s,t) & =-c\varphi'(s-ct)- \kap (\widecheck{a}^+*\varphi)(s-ct)+m\varphi(s-ct)\\&\quad +\kn\varphi(s-ct) (\widecheck{a}^-*\varphi)(s-ct) +\kl\varphi^2(s-ct),
\end{align*}
hence, to prove \eqref{supersol}, it is enough to show that, for all $s\in\R$,
\begin{align}\notag
  \mathcal{J}_c(s):&=c\varphi'(s)+\kap (\widecheck{a}^+*\varphi)(s)-m\varphi(s) \\&\quad- \kn\varphi(s)(\widecheck{a}^-*\varphi)(s) - \kl \varphi^2(s) \leq 0.\label{suffcond}
\end{align}

By \eqref{defvarphi}, \eqref{acheckpos}, for $s<0$, we have
\begin{align*}
  \mathcal{J}_c(s)&=\kap (\widecheck{a}^+*\varphi)(s)-m\theta-\kn\theta(\widecheck{a}^-*\varphi)(s) - \kl \theta^2 \\
  &\leq \bigl((\kap \widecheck{a}-\kn\theta\widecheck{a}^-)*\theta\bigr)(s) -m\theta -\kl\theta^2=0.
\end{align*}
Next, by \eqref{defvarphi},
\[
(\widecheck{a}^+*\varphi)(s)\leq\theta \int_\R \widecheck{a}^+(\tau)e^{-\mu(s-\tau)}\,d\tau=\theta e^{-\mu s} \A_\xi(\mu),
\]
therefore, for $s\geq0$, we have
\[
    \mathcal{J}_c(s)\leq -\mu c\theta e^{-\mu s}
    +\kap \theta  e^{-\mu s} \A_\xi(\mu) -m\theta e^{-\mu s};
\]
and to get \eqref{suffcond} it is enough to demand that $\kap  \A_\xi(\mu)-m- \mu c\leq0$, in particular,
\begin{equation}\label{demand}
c=\frac{\kap  \A_\xi(\mu)-m}{\mu}.
\end{equation}
As a result, for $\bar\phi(s,t)=\varphi(s-ct)$ with $c$ given by \eqref{demand}, we have
\begin{equation}\label{supersol2}
\mathcal{F}\bar\phi\geq0=\mathcal{F}(\widetilde{Q}_t\varphi),
\end{equation}
as $\widetilde{Q}_t\varphi$ is a solution to \eqref{eq:basic_one_dim}. Then, by \eqref{as:aplus_gr_aminus} and the inequality $\bar\phi\leq\theta$, one can apply Proposition~\ref{compprabscont} and get that
\[
\widetilde{Q}_t\varphi(s')\leq \bar\phi(t,s')=\varphi(s'-ct), \quad \text{a.a.}\ s'\in\R,
\]
where $c$ is given by \eqref{demand}; note that, by \eqref{defvarphi}, for any $s\in\R$, the function $\bar{\phi}(s,t)$ is absolutely continuous in $t$. In particular, for $t=1$, $s'=s+c$, we get
\begin{equation}\label{mayYag}
\widetilde{Q}_1\varphi(s+c)\leq \varphi(s), \quad \text{a.a.}\ s\in\R.
\end{equation}
And now one can apply \cite[Theorem 5]{Yag2009} which states that, if there exists a flow of abstract mappings $\widetilde{Q}_t$, each of them maps $\M$ into itself and has properties \ref{eq:QBtheta_subset_Btheta}--\ref{prop:Q_cont} of Theorem~\ref{thm:Qholds}, and if, for some $t$ (e.g. $t=1$), for some $c\in\R$, and for some $\varphi\in\M$, the inequality \eqref{mayYag} holds, then there exists $\psi\in\M$ such that, for any $t\geq0$,
\begin{equation}\label{getbyYag}
(\widetilde{Q}_t \psi)(s+ct)=\psi(s), \quad \text{a.a.}\ s\in\R,
\end{equation}
that yields the solution to \eqref{eq:basic_one_dim} in the form \eqref{tw1d}, and hence, by Remark~\ref{rem:multi-one}, we will get the existence of a solution to \eqref{eq:basic} in the form \eqref{eq:deftrw}. It is worth noting that, in \cite{Yag2009}, the results were obtained for increasing functions. By~Remark~\ref{incrinsteadofdecr}, the same results do hold for decreasing functions needed for our settings.

Next, by \cite[Theorem 6]{Yag2009}, there exists $c_*=c_*(\xi)\in(-\infty,\infty]$ such that, for any $c\geq c_*$, there exists $\psi=\psi_c\in\M$ such that \eqref{getbyYag} holds, and for any $c<c_*$ such a $\psi$ does not exist. Since for $c$ given by \eqref{demand} such a $\psi$ exists, we have that $c_*\leq c<\infty$, moreover, one can take any $\mu$ in \eqref{demand} for that \eqref{aplusexpint1} holds. Therefore,
\begin{equation}\label{cstarestimate}
c_*\leq \inf_{\la>0}\frac{\kap  \A_\xi(\la)-m}{\la}.
\end{equation}
The statement is proved.
\end{proof}

\begin{remark}
  It can be seen from the proof above that we didn't use the special form \eqref{defvarphi} of the function $\varphi$ after the inequality \eqref{supersol2}. Therefore, if a function $\varphi_1\in\M$ is such that the function $\bar\phi(s,t):=\varphi_1(s-ct)$, $s\in\R$, $t\geq0$, is a super-solution to \eqref{eq:basic_one_dim}, for some $c\in\R$, i.e. if \eqref{supersol} holds, then there exists a traveling wave solution to \eqref{eq:basic_one_dim}, and hence to \eqref{eq:basic}, with some profile $\psi\in\M$ and the same speed $c$.
\end{remark}

We are going to prove now the second item of Theorem~\ref{thm:trwexist}. We start with the following
\begin{proposition}\label{prop:reg_trw}
Let $\psi\in\M$ and $c\in\R$ be such that there exists a solution $u\in\Y_\infty$ to the equation \eqref{eq:basic} such that \eqref{eq:deftrw} holds, for some $\xi\in S^{d-1}  $. Then $\psi\in C^{1}(\R\to[0,\theta])$, for $c\neq0$, and $\psi\in C(\R\to[0,\theta])$, otherwise.
\end{proposition}
\begin{proof}
The condition \eqref{eq:deftrw} implies \eqref{trwvincond} for the $\xi\in S^{d-1}  $. Then, by Proposition~\ref{prop:monot_sol}, there exists $\phi$ given by \eqref{repres} which solves \eqref{eq:basic_one_dim}; moreover, by Remark~\ref{rem:multi-one}, \eqref{tw1d} holds.

Let $c\neq0$. It is well-known that any monotone function is differentiable almost everywhere. Prove first that $\psi$ is differentiable everywhere on $\R$. Fix any $s_{0}\in\R$.
It follows directly from Proposition~\ref{prop:monot_sol}, that $\phi\in C^1((0,\infty)\to L^\infty(\R))$. Therefore, for any $t_0>0$ and for any $\varepsilon>0$, there exists $\delta=\delta(t_0,\varepsilon)>0$
such that, for all $t\in\R$ with $|ct|<\delta$ and $t_0+t>0$, the following inequalities hold, for a.a.~$s\in\R$,
\begin{gather}
\dfrac{\partial \phi}{\partial t}(s,t_{0}) -\eps< \dfrac{\phi(s,t_{0}+t)-\phi(s,t_{0})}{t}<\dfrac{\partial \phi}{\partial t}(s,t_{0})+\varepsilon,\label{firstfromeq}\\
\dfrac{\partial \phi}{\partial t}(s,t_{0})-\eps<\dfrac{\partial \phi}{\partial t}(s,t_{0}+t)<\dfrac{\partial \phi}{\partial t}(s,t_{0})+\varepsilon.\label{secondfromeq}
\end{gather}

Set, for the simplicity of notations, $x_0=s_0+ct_0$. Take any $0<h<1$ with $2h<\min\bigl\{\delta,|c|t_0, |c|\delta \bigr\}$.
Since $\psi$ is a decreasing function, one has, for almost all $s\in(x_0,x_0+h^{2})$,
\begin{align}
&\quad \dfrac{\psi(s_0+h)-\psi(s_0)}{h}\leq\dfrac{\psi(s-ct_0+h-h^{2})-\psi(s-ct_0)}{h} \nonumber \\
&=\dfrac{\phi(s,t_0+\frac{h^{2}-h}{c})-\phi(s,t_0)}{\frac{h^{2}-h}{c}}\dfrac{h^{2}-h}{ch}\leq\left(\dfrac{\partial \phi}{\partial t}(s,t_0)\mp\varepsilon\right)\dfrac{h-1}{c},\label{ff1}
\end{align}
by \eqref{firstfromeq} with $t=\frac{h^{2}-h}{c}$; note that then $|ct|=h-h^2<h<\delta$, and $t_0+t>0$ (the latter holds, for $c<0$, because of $t_0+t>t_0$ then; and, for $c>0$, it is equivalent to $ct_0>-ct=h-h^2$, that follows from $h<ct_0$).
Stress, that, in \eqref{ff1}, one needs to choose $-\eps$, for $c>0$, and $+\eps$, for $c<0$, according to the left and right inequalities in \eqref{firstfromeq}, correspondingly.

Similarly, for almost all $s\in(x_0-h^{2},x_0)$, one has
\begin{align}
&\quad \dfrac{\psi(s_0+h)-\psi(s_0)}{h}\geq\dfrac{\psi(s-ct_0+h+h^{2})-\psi(s-ct_0)}{h} \nonumber\\
&=\dfrac{\phi(s,t_0-\frac{h^{2}+h}{c})-\phi(s,t_0)}{-\frac{h^{2}+h}{c}}\dfrac{h^{2}+h}{-ch}\geq\left(\dfrac{\partial \phi}{\partial t}(s,t_0)\pm\varepsilon\right)\dfrac{h+1}{-c},\label{ff2}
\end{align}
where we take again the upper sign, for $c>0$, and the lower sign, for $c<0$; note also that $h+h^2<2h<\delta$.
Next, one needs to `shift' values of $s$ in \eqref{ff2} to get them the same as in \eqref{ff1}. To do this note that, by \eqref{tw1d},
\begin{equation}\label{ff3}
\phi\Bigl(s+h^2,t_0+\frac{h^2}{c}\Bigr)=\phi(s,t_0), \quad \text{a.a.}\ s\in\X.
\end{equation}
As a result,
\begin{equation}\label{ff4}
\begin{split}(\widecheck{a}^\pm *\phi)\Bigl(s+h^2,t_0+\frac{h^2}{c}\Bigr)&=\int_{\R} \widecheck{a}^\pm (s') \phi\Bigl(s-s'+h^2,t_0+\frac{h^2}{c}\Bigr)\,ds\\
&=(\widecheck{a}^\pm *\phi)(s,t_0), \quad \text{a.a.}\ s\in\X.
\end{split}
\end{equation}
Then, by \eqref{eq:basic_one_dim}, \eqref{ff3}, \eqref{ff4}, one gets
\begin{equation}\label{ff5}
\frac{\partial}{\partial t}\phi\Bigl(s+h^2,t_0+\frac{h^2}{c}\Bigr)=\frac{\partial}{\partial t}\phi(s,t_0), \quad \text{a.a.}\ s\in\X.
\end{equation}
Therefore, by \eqref{ff5}, one gets from \eqref{ff2} that, for almost all $s\in(x_0,x_0+h^{2})$, cf. \eqref{ff1},
\begin{align}
 \dfrac{\psi(s_0+h)-\psi(s_0)}{h}&\geq\left(\dfrac{\partial \phi}{\partial t}\Bigl(s,t_0+\frac{h^2}{c}\Bigr)\pm\varepsilon\right)\dfrac{h+1}{-c},\notag\\
 \intertext{and, since $\bigl\lvert \frac{h^2}{c}\bigr\rvert<\delta$, one can apply the right and left inequalities in \eqref{secondfromeq}, for $c>0$ and $c<0$, correspondingly, to continue the estimate}
& \geq\biggl(\dfrac{\partial \phi}{\partial t}(s,t_0)\pm 2\varepsilon\biggr)\dfrac{h+1}{-c}.\label{ff6}
\end{align}
Combining \eqref{ff1} and \eqref{ff6}, we obtain
\begin{multline}
\biggl(\esssup_{s\in(x_0,x_0+h^2)}\dfrac{\partial \phi}{\partial t}(s,t_0)\pm 2\varepsilon\biggr)\dfrac{h+1}{-c} \leq \dfrac{\psi(s_0+h)-\psi(s_0)}{h}\\ \leq\biggl(\esssup_{s\in(x_0,x_0+h^2)}\dfrac{\partial \phi}{\partial t}(s,t_0)\mp\varepsilon\biggr)\dfrac{h-1}{c}. \label{ff7}
\end{multline}
For fixed $s_0\in\R$, $t_0>0$ and for $x_0=s_0+ct_0$, the function
\[
	f(h):=\esssup\limits_{s\in(x_0,x_0+h^2)}\frac{\partial \phi}{\partial t}(s,t_0), \quad h\in(0,1),
\]
is bounded, as $|f(h)|\leq \bigl\lVert \frac{\partial \phi}{\partial t}(\cdot,t_0)\bigr\rVert_\infty<\infty$, and monotone; hence there exists $\bar f=\lim\limits_{h\to0+}f(h)$. As a result, for small enough $h$, \eqref{ff7} yields
\begin{equation*}
(\bar f\pm 2\varepsilon)\dfrac{1}{-c} -\eps \leq \dfrac{\psi(s_0+h)-\psi(s_0)}{h} \leq(\bar f\mp\varepsilon)\dfrac{-1}{c}+\eps,
\end{equation*}
and, therefore, there exists
$\dfrac{\partial\psi}{\partial s}(s_0+)=\dfrac{-\bar f}{c}$. In the same way, one can prove that there exists $\dfrac{\partial\psi}{\partial s}(s_0-)=\dfrac{-\bar f}{c}$, and, therefore, $\psi$ is differentiable at $s_0$.
As a result, $\psi$ is differentiable (and hence continuous) on the whole $\R$.

Next, for any $s_1,s_2,h\in\R$, we have
\begin{multline*}
  \biggl\lvert \frac{\psi(s_1+h)-  \psi(s_1)}{h}-\frac{\psi(s_2+h)-  \psi(s_2)}{h}\biggr\rvert\\
  =\frac{1}{|c|}\biggl\lvert \frac{\phi\bigl(s_1+ct_0,t_0-\frac{h}{c}\bigr)-  \phi(s_1+ct_0,t_0)}{-\frac{h}{c}}\qquad\qquad\qquad\\
  -\frac{\phi\bigl(s_1+ct_0,t_0+\frac{s_1-s_2}{c}-\frac{h}{c}\bigr)-  \phi\bigl(s_1+ct_0,t_0+\frac{s_1-s_2}{c}\bigr)}{-\frac{h}{c}}\biggr\rvert;
\end{multline*}
and if we pass $h$ to $0$, we get
\begin{align}\notag
  \lvert \psi'(s_1)-\psi'(s_2)\rvert&=\frac{1}{|c|}\biggl\lvert \frac{\partial}{\partial t}\phi(s_1+ct_0,t_0)
  -\frac{\partial}{\partial t}\phi\Bigl(s_1+ct_0,t_0+\frac{s_1-s_2}{c}\Bigr)\biggr\rvert
  \\& \leq \frac{1}{|c|}\biggl\lVert \frac{\partial}{\partial t}\phi(\cdot,t_0)
  -\frac{\partial}{\partial t}\phi\Bigl(\cdot,t_0+\frac{s_1-s_2}{c}\Bigr)\biggr\rVert.\label{eq333}
\end{align}
And now, by the continuity of $\frac{\partial}{\partial t}\phi(\cdot,t)$ in $t$ in the sense of the norm in $L^\infty(\R)$, we have that, by \eqref{secondfromeq}, the inequality $|s_1-s_2|\leq |c|\delta$ implies that, by \eqref{eq333},
$\lvert \psi'(s_1)-\psi'(s_2)\rvert\leq \frac{1}{|c|} \eps$.
As a result, $\psi'(s)$ is uniformly continuous on $\R$ and hence continuous.

Finally, consider the case $c=0$. Then \eqref{tw1d} implies that $\phi(s,t)$ must be constant in time, i.e. $\phi(s,t)=\psi(s)$, for a.a. $s\in\R$. Thus one can rewrite \eqref{eq:basic_one_dim} as follows
\begin{align}
  0 &= -\kap (\widecheck{a}^{+}*\psi)(s) +m\psi(s)+\kn\psi(s) (\widecheck{a}^{-}*\psi)(s) +\kl \psi^2(s) \nonumber \\ 
&= \kl\psi^2(s) + A(s) \psi(s) - B(s), \label{statwave}
\end{align}
where $A(s) = m+\kn (\widecheck{a}^{-}*\psi)(s)$ and $B(s) = \kap (\widecheck{a}^{+}*\psi)(s) $. Equivalently,

\begin{equation}\label{asquotient}
	\psi(s)= \frac{\sqrt{A^2(s)+4\kl B(s)}-A(s)}{4\kl}.
\end{equation}
Since $\psi\in L^\infty(\R)$, then, by Lemma~\ref{le:simple}, the r.h.s. of \eqref{asquotient} is a continuous in $s$ function, and hence $\psi\in C(\R)$.
\end{proof}

\begin{proposition}
Let $\psi\in\M$, $c\in\R$, $\xi\in S^{d-1} $ be such that there exists a solution $u\in\Y_\infty$ to the equation \eqref{eq:basic} such that \eqref{eq:deftrw} holds. Then, for each $s\in\R$,
\begin{multline}
  c\psi'(s) +\kap (\widecheck{a}^{+}*\psi)(s) -m\psi(s)
    -\kn\psi(s) (\widecheck{a}^{-}*\psi)(s) -\kl \psi^2(s)=0.\label{eq:trw}
\end{multline}
\end{proposition} 
\begin{proof}
Let $c\neq0$. Then, by~Remark~\ref{rem:multi-one} and Proposition~\ref{prop:reg_trw},  one can differentiate $\psi(s-ct)$ in $t\geq0$. By this and Lemma~\ref{le:simple} we get \eqref{eq:trw} for all $s\in\R$. For $c=0$, one has \eqref{statwave}, i.e. \eqref{eq:trw} holds in this case as well.
\end{proof}

Let $k\in\N\cup\{\infty\}$ and $C_b^k(\R)$ denote the class of all functions on $\R$ which are $k$ times differentiable and whose derivatives (up to the order $k$) are continuous and bounded on $\R$. The following corollary finishes the proof of the second item of Theorem~\ref{thm:trwexist}.
\begin{corollary}\label{cor:infsmoothprofile}
Let $\psi\in\M$, $c\in\R$, $c\neq0$, $\xi\in S^{d-1} $ be such that there exists a solution $u\in\Y_\infty$ to the equation \eqref{eq:basic} such that \eqref{eq:deftrw} holds. Then $\psi\in C_b^\infty(\R)$.
\end{corollary}
\begin{proof}
By Lemma~\ref{le:simple}, $\widecheck{a}^\pm*\psi\in C_b(\R)$. Then \eqref{eq:trw} yields $\psi'\in C_b(\R)$, i.e. $\psi\in C_b^1(\R)$. By e.g. \cite[Proposition~5.4.1]{Sta2005}, $\widecheck{a}^\pm*\psi\in C_b^1(\R)$ and $(\widecheck{a}^\pm*\psi)'=\widecheck{a}^\pm*\psi'$, therefore, the equality \eqref{eq:trw} holds with $\psi'$ replaced by $\psi''$ and $\psi$ replaced by $\psi'$. Then, by the same arguments $\psi\in C_b^2(\R)$, and so on. The statement is proved.
\end{proof}

We are going to prove now the third item of Theorem~\ref{thm:trwexist}. We will follow ideas of \cite{CD2007}.
\begin{proposition}\label{prop:trw_exp_est}
Let \eqref{as:chiplus_gr_m} and \eqref{as:aplus_gr_aminus-intro} hold. Let $\psi\in\M$, $c\in\R$, $\xi\in S^{d-1} $ be such that there exists a solution $u\in\Y_\infty$ to the equation \eqref{eq:basic} such that \eqref{eq:deftrw} holds. Then there exists $\mu=\mu( c, a^+, \kam, \theta)>0$ such that
  $\int_\R\psi(s)e^{\mu s}\,ds<\infty$.
\end{proposition}
\begin{proof}
At first, we prove that $\psi\in L^1(\R_+)$. 
Under assumptions \eqref{as:chiplus_gr_m} and \eqref{as:aplus_gr_aminus-intro}, define the following function:
\begin{equation}\label{speckern}
  \widecheck{J}_\upsilon(s):=\kap \widecheck{a}^+(s)-\upsilon\kn\widecheck{a}^-(s), \quad s\in\R, \upsilon\in(0,\theta].
\end{equation}
Then, by \eqref{acheckpos}, $\widecheck{J}_\upsilon(s)\geq \widecheck{J}_\theta(s)\geq0$ for $s\in\R$, $\upsilon\in(0,\theta]$.
Since $\int_\R \widecheck{J}_\upsilon(s)\,ds = \kap -\upsilon\kn > m+\kl \upsilon$, one can choose $R_0>0$, such that
\begin{equation}\label{Risproper}
	\int_{-R_0}^{R_0} \widecheck{J}_\upsilon(s)\,ds = m+\kl\upsilon.
\end{equation}
We rewrite \eqref{eq:trw} as follows
\begin{multline}\label{eq:tr_w_ii}
	c\psi '(s) + (\widecheck{J}_\upsilon*\psi)(s) + \bigl(\upsilon-\psi(s)\bigr)\big(\kl\psi(s) +\kn(\widecheck{a}^-*\psi)(s)\big) \\
		- (m+\kl\upsilon)\psi(s) = 0,\quad s\in\R.
\end{multline}
Fix arbitrary ${r_0}>0$, such that
\begin{equation}\label{rhocond}
  \psi({r_0})<\upsilon.
\end{equation}
Let $r>{r_0}+R_0$. Integrate \eqref{eq:tr_w_ii} over $[{r_0},r]$; one gets
\begin{equation}\label{eq:tr_w_ii_int}
c(\psi(r)-\psi({r_0}))+A+B=0,
\end{equation}
where
\begin{align*}
A&:= \int_{{r_0}}^{r}(\widecheck{J}_\upsilon*\psi)(s)\, ds - (m+\kl\upsilon)\int_{{r_0}}^{r}\psi(s)ds,\\
B&:= \int_{{r_0}}^{r}(\upsilon-\psi(s))\bigl(\kl\psi(s) + \kn(\widecheck{a}^-*\psi)(s)\bigl)\,ds.
\end{align*}
By \eqref{speckern}, \eqref{Risproper}, one has
\begin{align}
	A& \geq \int_{r_0}^{r} \int_{-R_0}^{R_0} \widecheck{J}_\upsilon(\tau) \psi(s-\tau)d\tau ds -(m+\kl\upsilon) \int_{r_0}^{r} \psi(s)\,ds\nonumber\\
	&= \int_{-R_0}^{R_0} \widecheck{J}_\upsilon(\tau) \left( \int_{{r_0}-\tau}^{r-\tau}\psi(s)\,ds - \int_{{r_0}}^{r}\psi(s)\,ds \right)\,d\tau\nonumber \\
	&=\int_{0}^{R_0}\widecheck{J}_\upsilon(\tau)\left( \int_{{r_0}-\tau}^{{r_0}}\psi(s)\,ds-\int_{r-\tau}^{r}\psi(s)\,ds \right)\,d\tau\nonumber\\
	&\quad+\int_{-R_0}^{0}\widecheck{J}_\upsilon(\tau)\left( \int_{r}^{r-\tau}\psi(s)\,ds-\int_{{r_0}}^{{r_0}-\tau}\psi(s)\,ds \right)\,d\tau; \label{eq:gen_est12}
\end{align}
and since $\psi$ is a decreasing function and $r-R_0>{r_0}$, we have from \eqref{eq:gen_est12}, that
\begin{align}
 A &\geq (\psi({r_0})-\psi(r-R_0))\int_{0}^{R_0}\tau \widecheck{J}_\upsilon(\tau)\,d\tau+(\psi(r+R_0)-\psi({r_0}))\int_{-R_0}^{0}(-\tau) J_\upsilon(\tau)\,d\tau \notag \\ 
	&\geq -\theta \int_{-R_0}^{0}(-\tau) J_\upsilon(\tau)\,d\tau =:-\theta \bar{J}_{\upsilon,R_0}. \label{eq:gen_est}
\end{align}
Next, \eqref{rhocond} and monotonicity of $\psi$ imply
\begin{equation}\label{B-est}
  B\geq (\upsilon-\psi({r_0})) \int_{{r_0}}^{r} \bigl(\kl\psi(s) + \kn(\widecheck{a}^-*\psi)(s)\bigl) \,ds.
\end{equation}
Then, by \eqref{eq:tr_w_ii_int}, \eqref{eq:gen_est}, \eqref{B-est}, \eqref{rhocond}, one gets
\begin{align*}
  0 \leq (\upsilon-\psi({r_0})) &\int_{{r_0}}^{r} \bigl(\kl\psi(s) + \kn(\widecheck{a}^-*\psi)(s)\bigl) \,ds \\
  &\leq \theta \bar{J}_{\upsilon,R_0} + c(\psi({r_0})-\psi(r)) \to \theta \bar{J}_{\upsilon,R_0} + c\psi({r_0})<\infty, \quad r\to\infty,
\end{align*}
therefore, $\kl\psi + \kn\widecheck{a}^-*\psi\in L^1(\R_+)$. Finally, \eqref{cleareq} implies that there exist a measurable bounded set $\Delta\subset\R$, with $m(\Delta):=\int_\Delta \,ds\in (0,\infty)$, and a constant $\mu>0$, such that $\widecheck{a}^-(\tau)\geq\mu$, for a.a. $\tau\in\Delta$. Let $\delta=\inf \Delta\in\R$. Then, for any $s\in\R$, one has
\begin{equation*}
(\widecheck{a}^-*\psi)(s)\geq \int_\Delta \widecheck{a}^-(\tau) \psi(s-\tau)\,d\tau\geq \mu \psi(s-\delta) m(\Delta).
\end{equation*}
Therefore $\psi\in L^1(\R_+)$.

For any $N\in\N$, we define $\varphi_N(s):=\1_{(-\infty,N)}(s)+e^{-\la(s-N)}\1_{[N,\infty)}(s)$, where $\la>0$. By the proved above, $\psi,\widecheck{a}^\pm*\psi\in L^1(\R_+)\cap L^\infty(\R)$ hence, by \eqref{eq:trw}, $c\psi'\in L^1(\R_+)\cap L^\infty(\R)$. Therefore, all terms of \eqref{eq:trw} being multiplied on $e^{\lambda s}\varphi_{N}(s)$ are~integrable over $\R$. After this integration, \eqref{eq:trw} will be read as follows
\begin{equation}
I_1+I_2+I_3=0,\label{eq:int_trw_zeta_exp}
\end{equation}
where (recall that $\kam \theta-\kap =-m$)
\begin{align*}
I_1&:=c\int_{\R}\psi' (s) e^{\lambda s}\varphi_{N}(s)\,ds,\\
I_2&:=\kap \int_\R\bigl((\widecheck{a}^{+}*\psi)(s)-\psi(s)\bigr)e^{\lambda s}\varphi_{N}(s)\,ds,\\
I_3&:= \int_{\R}\psi(s)\bigl(\kap-m-\kl\psi(s)-\kn(\widecheck{a}^{-}*\psi)(s)\bigr)
e^{\lambda s}\varphi_{N}(s)\,ds
\end{align*}
We estimate now $I_1,I_2,I_3$ from below.

We start with $I_2$. One can write
\begin{align}
&\quad\int_{\R}(\widecheck{a}^{+}*\psi)(s)e^{\lambda s}\varphi_{N}(s)\,ds
=\int_{\R}\int_{\R}\widecheck{a}^{+}(s-\tau)\psi(\tau)e^{\lambda s}\varphi_{N}(s)\,d\tau ds\notag\\
&=\int_{\R}\int_{\R}\widecheck{a}^{+}(s)e^{\lambda s}\varphi_{N}(\tau+s)\,ds\, e^{\lambda \tau}\psi(\tau)\,d\tau\notag\\
&\ge\int_{\R}\biggl(\int_{-\infty}^{R}\widecheck{a}^{+}(s)e^{\lambda s}\,ds\biggr)\varphi_{N}(\tau+R)e^{\lambda \tau}\psi(\tau)\,d\tau,\label{eq:111}
\end{align}
for any $R>0$, as $\varphi$ is nonincreasing. By \eqref{cleareq}, one can choose $R>0$ such that
\[
\int_{-\infty}^{R}\widecheck{a}^{+}(\tau)\,d\tau>1-\dfrac{\kam \theta}{4}.
\]
By continuity arguments, there exists $\nu>0$ such that, for any $0<\la<\nu$,
\begin{equation}\label{eq:222}
  \int_{-\infty}^{R}\widecheck{a}^{+}(\tau)e^{\lambda \tau}\,d\tau\geq\Bigl(1-\dfrac{\kam \theta}{4}\Bigr)e^{\lambda R}.
\end{equation}
Therefore, combining \eqref{eq:111} and \eqref{eq:222}, we get
\begin{align}
I_2&\geq\int _{\R}\Bigl(1-\dfrac{\kam \theta}{4}\Bigr)e^{\lambda R}\varphi_{N}(\tau+R)e^{\lambda \tau}\psi(\tau)\,d\tau-\int _{\R}\psi(s)e^{\lambda s}\varphi_{N}(s)\,ds\notag \\
&=\int_{\R}\Bigl(1-\dfrac{\kam \theta}{4}\Bigr)\varphi_{N}(\tau)e^{\lambda \tau}\psi(\tau-R)\,d\tau-\int _{\R}\psi(s)e^{\lambda s}\varphi_{N}(s)\,ds\notag \\
&\ge-\dfrac{\kam \theta}{4}\int_{\R}\psi(s)e^{\lambda s}\varphi_{N}(s)\,ds,\label{eq:trw_exp_est:i}
\end{align}
as $\psi(\tau-R)\geq\psi(\tau)$, $\tau\in\R$, $R>0$.

Now we estimate $I_3$. By \eqref{eq:deftrw}, it is easily seen that the function $(\widecheck{a}^{-}*\psi)(s)$ decreases monotonically to $0$ as $s\to\infty$.
Suppose additionally that $R>0$ above is such that
\[
	\kl\psi(s) + \kn(\widecheck{a}^{-}*\psi)(s)<\dfrac{\kam\theta}{2}, \quad s>R.
\]
Then, one gets
\begin{align*}
	I_3& \geq \dfrac{\kam\theta}{2} \int_{R}^{\infty} \psi(s)e^{\lambda s}\varphi_{N}(s)\,ds \nonumber \\ 
		&\qquad + \int_{-\infty}^{R}\psi(s)\bigl(\kam\theta-\kl\psi(s)-\kn(\widecheck{a}^{-}*\psi)(s)\bigr)e^{\lambda s}\varphi_{N}(s)\,ds \nonumber \\
		&\geq \dfrac{\kam\theta}{2}\int _{R}^{\infty}\psi (s)e^{\lambda s}\varphi_{N}(s)\,ds,
\end{align*}
as $0\leq\psi\leq\theta$, $\varphi_N\geq0$, $(\widecheck{a}^{-}*\psi)(s)\leq\theta$.

It remains to estimate $I_1$ (in the case $c\neq0$). Since
$\lim\limits_{s\to\pm\infty} \psi(s)e^{\la s}\varphi_N(s) =0$, we have from the integration by parts formula, that
\begin{equation*}
I_1=-c\int_{\R}\psi(s)(\lambda\varphi_{N}(s)+\varphi_{N}'(s))e^{\lambda s}\,ds.
\end{equation*}
For $c>0$, one can use that $\varphi_N'(s)\leq0$, $s\in\R$, and hence
\begin{equation*}
  I_1\geq -c \lambda\int _{\R}\psi(s)\varphi_{N}(s)e^{\lambda s}\,ds.
\end{equation*}
For $c<0$, we use that, by the definition of $\varphi_N$, $\lambda\varphi_{N}(s)+\varphi_{N}'(s)=0$, $s\geq N$; therefore,
\begin{equation}
I_1=-c\lambda\int_{-\infty}^N\psi(s)\,ds>0. \label{eq:trw_exp_est:iv}
\end{equation}

Therefore, combining \eqref{eq:trw_exp_est:i}--\eqref{eq:trw_exp_est:iv}, we get from \eqref{eq:int_trw_zeta_exp}, that
\begin{equation*}
0\geq
-\lambda \bar{c}\int _{\R}\psi(s)\varphi_{N}(s)e^{\lambda s}\,ds -\dfrac{\kam \theta}{4}\int_{\R}\psi(s) e^{\lambda s}\varphi_{N}(s)\,ds +\dfrac{\kam \theta}{2}\int_{R}^{\infty}\psi(s) e^{\lambda s}\varphi_{N}(s)\,ds,
\end{equation*}
where $\bar{c}=\max\{c,0\}$.

The latter inequality can be easily rewritten as
\begin{align}\notag
&\quad \Bigl(\dfrac{\kam \theta}{4}-\lambda \bar{c}\Bigr)\int_{R}^{\infty}\psi (s) e^{\lambda s}\varphi_{N}(s)\,ds\leq \Bigl(\dfrac{\kam \theta}{4} +\lambda \bar{c}\Bigr)\int_{-\infty}^{R}\psi(s)\varphi_{N}(s)e^{\lambda s}\,ds\\
&\leq \Bigl(\dfrac{\kam \theta}{4}+\lambda \bar{c}\Bigr)\theta \int_{-\infty}^{R}e^{\lambda s}\,ds=:I_{\la,R}<\infty, \qquad 0<\la<\nu. \label{eq:333}
\end{align} 

Take now $\mu<\min\bigl\{\nu, \frac{\kam \theta}{4c}\bigr\}$, for $c>0$, and $\mu<\nu$, otherwise. Then, by \eqref{eq:333}, for any $N>R$, one get
\[
\infty>\Bigl(\dfrac{\kam \theta}{4}-\mu \bar{c}\Bigr)^{-1}I_{\mu,R}>\int_{R}^{\infty}\psi (s) e^{\mu s}\varphi_{N}(s)\,ds\geq
\int_{R}^{N}\psi (s) e^{\mu s}\,ds,
\]
thus,
\begin{align*}
\int_{\R}\psi (s) e^{\mu s}\,ds&=\int_{-\infty}^R \psi (s) e^{\mu s}\,ds
+\int_{R}^{\infty}\psi (s) e^{\mu s}\,ds
\\&\leq \theta \int_{-\infty}^R e^{\mu s}\,ds+I_{\mu,R}\Bigl(\dfrac{\kam \theta}{4}-\mu \bar{c}\Bigr)^{-1}<\infty,
\end{align*}
that gets the statement.
\end{proof}

By Proposition~\ref{prop:reg_trw}, a traveling wave solution to \eqref{eq:basic} is continuous in space as well. Because of this, to prove the fourth item of Theorem~\ref{thm:trwexist}, we can use the strong maximum principle. We suppose that $a^+$ is not degenerated in the direction $\xi$ at the origin, namely, there exist $r\geq0$, $\rho,\delta>0$ (depending on $\xi$), such that
						\begin{assum}\label{as:a+nodeg1d}
    					\int_{\{x\cdot\xi=s\}}a^{+}(x)\,dx \geq\rho \text{ \ for a.a. } |s|\leq \delta.
    				\end{assum}
Clearly, either of \eqref{as:aplus-aminus-is-pos1d}, \eqref{as:aplus-aminus-is-pos} or \eqref{as:a+nodeg} implies \eqref{as:a+nodeg1d}.

\begin{proposition}\label{prop:psidecaysstrictly}
 Let \eqref{as:chiplus_gr_m},  \eqref{as:aplus_gr_aminus-intro} and \eqref{as:a+nodeg1d} hold. Let $\psi\in\M$, $c\in\R$, $\xi\in S^{d-1} $ be such that there exists a solution $u\in\Y_\infty$ to the equation \eqref{eq:basic} such that \eqref{eq:deftrw} holds.
Then $\psi$ is a strictly decaying function, for any speed $c$.
\end{proposition}
\begin{proof}
By Remark~\ref{rem:multi-one}, there exists a traveling wave solution $\phi(s,t)=\psi(s-ct)$ to the equation \eqref{eq:basic_one_dim}.  By~Proposition~\ref{prop:reg_trw}, $\psi\in C(\R)$ and hence $\phi(s,t)=\psi(s-ct)$ is continuous in $s$ as well.
Suppose that $\psi$ is not strictly decaying, then there exists $\delta_0>0$ and $s_0\in\R$, such that $\psi(s)=\psi(s_0)$, for all $|s-s_0|\leq\delta_0$. Take any $\delta\in\bigl(0,\frac{\delta_0}{2}\bigr)$, and consider the function $\psi^\delta(s):=\psi(s+\delta)$. Clearly, $\psi^\delta(s)\leq\psi(s)$, $s\in\R$.
By~Remarks~\ref{shiftoftrw}, \ref{rem:multi-one}, $\psi^\delta$ is a profile for a traveling wave solution to the equation \eqref{eq:basic_one_dim} with the same speed $c$. Therefore, one has two solutions to \eqref{eq:basic_one_dim}: $\phi(s,t)=\psi(s-ct)$ and $\phi^\delta(s,t)=\psi^\delta(s-ct)$ and hence
  $\phi^\delta(s,t)\leq \phi(s,t)$, $s\in\R$, $t\geq0$.
  By the maximum principle for the equation \eqref{eq:basic_one_dim}, see Theorem~\ref{thm:strongmaxprinciple} with $d=1$, either $\phi\equiv \phi^\delta$, that contradicts $\delta>0$ or $\phi^\delta(s,t)< \phi(s,t)$, $s\in\R$, $t>0$. The latter, however, contradicts the equality $\phi^\delta(s,t)=\phi(s,t)$, which holds e.g.\! for $s=s_0+ct$, $ct< \delta_0$. Hence $\psi$ is a strictly decaying function. 
\end{proof}

To prove the last item of Theorem~\ref{thm:trwexist}, one can weaken the assumption \eqref{as:a+nodeg1d}, assuming that $a^+$ is not degenerated in the direction $\xi$ (not necessarily at the origin). Namely, we assume that there exist $r\geq0$, $\rho,\delta>0$ (depending on $\xi$), such that
			\begin{assum}\label{as:a+nodeg-mod}
    					\int_{\{x\cdot\xi=s\}}a^{+}(x)\,dx \geq\rho \text{ \ for a.a.\! $s\in [r- \delta, r+ \delta]$}.
    	\end{assum}
\begin{proposition}\label{prop:trw_willbe_incr}
 Let \eqref{as:chiplus_gr_m},  \eqref{as:aplus_gr_aminus-intro} and \eqref{as:a+nodeg-mod} hold. Let $\psi\in\M$, $c\in\R$, $c\neq0$, $\xi\in S^{d-1} $ be such that there exists a solution $u\in\Y_\infty$ to the equation \eqref{eq:basic} such that \eqref{eq:deftrw} holds.  Then there exists $\nu>0$, such that $\psi(t)e^{\nu t}$ is a strictly increasing function.
\end{proposition}
\begin{proof}
We start from the case $c>0$. Since $\psi(t)>0,\ t\in\R$, it is sufficient to prove that
  \begin{equation}\label{weneed}
    \frac{\psi'(t)}{\psi(t)}> -\nu,\quad t\in\R.
  \end{equation}
  Fix any $\mu\geq\dfrac{\kap }{c}>0$. Then, clearly,
  \[
  	\kl\psi^2(t) + \kn(\widecheck{a}^- *\psi)(t)+m\leq \kam \theta+m=\kap \leq c\mu,
  \]
  and we will get from \eqref{eq:trw}, that
  \begin{equation}
0\geq c\psi'(s)+\kap (\widecheck{a}^+ *\psi)(s)-c\mu\psi(s), \quad s\in\R.
\label{ineq1}
\end{equation}
Multiply both parts of \eqref{ineq1} on $e^{-\mu s}>0$ and set
\[
w(s):=\psi(s)e^{-\mu s}>0, \quad s\in\R.
\]
Then $w'(s)=\psi'(s)e^{-\mu s}-\mu w(s)$ and one can rewrite \eqref{ineq1} as follows
  \begin{align}
0&\geq c w'(s)+\kap (\widecheck{a}^+ *\psi)(s)e^{-\mu s}\notag\\&=c w'(s)+\kap \int_\R\widecheck{a}^+ (\tau)w(s-\tau)e^{-\mu \tau}d\tau, \quad s\in\R.
\label{eq:w_est}
\end{align}

By \eqref{as:a+nodeg-mod}, there exists $\varrho:=\frac{r}{2}+\frac{\delta}{4}>0$, such that
\begin{equation}\label{intfrom2rhoispos}
\int_{2\varrho}^{\infty}\widecheck{a}^+ (s)e^{-\mu s}ds>0.
\end{equation}
Integrating \eqref{eq:w_est} over $s\in[t,t+\varrho]$, one gets
  \begin{equation}\label{asd3}
  0 \geq c(w(t+\varrho)-w(t))+\kap \int_{t}^{t+\varrho}\int_\R\widecheck{a}^+ (\tau)w(s-\tau)e^{-\mu \tau}d\tau ds.
\end{equation}
Since $w(t)$ is a monotonically decreasing function, we have
  \begin{align}
     \int_{t}^{t+\varrho}\int_\R\widecheck{a}^+ (\tau)w(s-\tau)e^{-\mu \tau}d\tau ds     &\geq \varrho\int_\R\widecheck{a}^+ (\tau)w(t+\varrho-\tau)e^{-\mu \tau}d\tau \nonumber \\
      \geq \varrho \int_{2\varrho}^{\infty}\widecheck{a}^+ (\tau)w(t+\varrho-\tau)e^{-\mu \tau}d\tau &\geq \varrho w(t-\varrho)\int_{2\varrho}^{\infty}\widecheck{a}^+ (\tau)e^{-\mu \tau}d\tau.\label{asd4}
  \end{align}
We set, cf. \eqref{intfrom2rhoispos},
\[
C(\mu,\rho):=\dfrac{\kap }{c}\int_{2\varrho}^{\infty}\widecheck{a}^+ (s)e^{-\mu s}ds>0.
\]
Then \eqref{asd3} and \eqref{asd4} yield
  \begin{equation}\label{eq:ln_fy_est_i}
    w(t) -\varrho C(\mu,\rho)w(t-\varrho)\geq w(t+\varrho)>0,\quad t\in\R.
  \end{equation}

  Now we integrate \eqref{eq:w_est} over $s\in[t-\varrho,t]$. Similarly to above, one gets
  \begin{align}
    0 &\geq c(w(t)-w(t-\varrho))+\kap \int_{t-\varrho}^{t}\int_\R\widecheck{a}^+ (\tau)w(s-\tau)e^{-\mu \tau}d\tau ds \nonumber \\
      &\geq c(w(t)-w(t-\varrho))+\varrho\kap \int_\R\widecheck{a}^+ (\tau)w(t-\tau)e^{-\mu \tau}d\tau.\label{asda12}
  \end{align}
  By \eqref{eq:ln_fy_est_i} and \eqref{asda12}, we have
  \begin{equation}\label{eq:ln_fy_est_ii}
  \frac{1}{\varrho C(\mu,\rho)}\geq \dfrac{ w(t-\varrho)}{w(t)}\geq 1 +\dfrac{\varrho\kap }{c}\int_\R\widecheck{a}^+ (\tau)\dfrac{w(t-\tau)}{w(t)}e^{-\mu \tau}d\tau.
  \end{equation}
  On the other hand, \eqref{eq:trw} implies that
  \begin{equation}
-\frac{\psi'(t)}{\psi(t)}\leq \frac{\kap }{c}\frac{(\widecheck{a}^+ *\psi)(t)}{\psi(t)}=\dfrac{\kap }{c}\int_\R\widecheck{a}^+ (\tau)\dfrac{w(t-\tau)}{w(t)}e^{-\mu \tau}d\tau, \quad t\in\R.\label{eq:ln_fy_est_ii11}
\end{equation}
Finally, \eqref{eq:ln_fy_est_ii} and \eqref{eq:ln_fy_est_ii11} yield \eqref{weneed} with $\nu=\dfrac{1}{\rho^2 C(\mu,\rho)}>0$.

Let now $c<0$. For any $\nu\in\R$, one has
\begin{equation*}
\psi'(s)=e^{-\nu s}(\psi(s)e^{\nu s})'-\nu \psi(s),\quad s\in\R.
\end{equation*}
Hence, by \eqref{eq:trw}, \eqref{as:aplus_gr_aminus-intro},
\begin{align*}
	0=& ce^{-\nu s}(\psi(s)e^{\nu s})'-c\nu\psi(s)+\kap(\widecheck{a}^+ *\psi)(s) \\ 
	 &\qquad \qquad \qquad \qquad \qquad \qquad -\kl\psi^2(s)-\kn\psi(s)(\widecheck{a}^- *\psi)(s)-m\psi(s)\nonumber \\
	\geq& ce^{-\nu s}(\psi(s)e^{\nu s})'-c\nu\psi(s) +\kap(\widecheck{a}^+ *\psi)(s) \\
	 &\qquad \qquad \qquad \qquad \qquad \qquad - \kl\theta\psi(s) -\kn\theta(\widecheck{a}^- *\psi)(s)-m\psi(s)\nonumber \\
	\geq& ce^{-\nu s}(\psi(s)e^{\nu s})'-c\nu\psi(s) -\kl \theta \psi(s) -m\psi(s),\quad s\in\R.
\end{align*}
As a result, choosing $\nu>\dfrac{m+\kl\theta}{-c}$, one gets
\begin{equation*}
-c e^{-\nu s}(\psi(s)e^{\nu s})'\geq(-c\nu-\kl\theta -m)\psi(s)>0,\quad s\in\R,
\end{equation*}
i.e. $\psi(s)e^{\nu s}$ is an increasing function.
\end{proof}
\color{black}

Combining Propositions \ref{prop:trwexists}, \ref{prop:reg_trw}, \ref{prop:trw_exp_est}--\ref{prop:trw_willbe_incr} and Corolalry~\ref{cor:infsmoothprofile}, we prove Theorem~\ref{thm:trwexist}.

\section*{Acknowledgments}
Authors gratefully acknowledge the financial support by the DFG through CRC 701 ``Stochastic
Dynamics: Mathematical Theory and Applications'' (DF, YK, PT), the European
Commission under the project STREVCOMS PIRSES-2013-612669 (DF, YK), and the ``Bielefeld Young Researchers'' Fund through the Funding Line Postdocs: ``Career Bridge Doctorate\,--\,Postdoc'' (PT).

\end{document}